\documentclass[10pt,reqno]{amsart}

\usepackage{color}
\usepackage{enumerate}
\usepackage{amssymb}
\usepackage{amsmath}
\usepackage{amsthm}
\usepackage{mathrsfs}

\newtheorem{theorem}{Theorem}[section]
\newtheorem{lemma}[theorem]{Lemma}
\newtheorem{corollary}[theorem]{Corollary}

\theoremstyle{remark}
\newtheorem{remark}[theorem]{Remark}

\theoremstyle{definition}

\newtheorem{definition}[theorem]{Definition}

\newcommand\cbrk{\text{$]$\kern-.15em$]$}}
\newcommand\opar{\text{\,\raise.2ex\hbox{${\scriptstyle|}$}\kern-.34em$($}}
\newcommand\cpar{\text{$)$\kern-.34em\raise.2ex\hbox{${\scriptstyle |}$}}\,}

\def\qed{{\hfill $\Box$ \bigskip}}

\def\XXint#1#2#3{{\setbox0=\hbox{$#1{#2#3}{\int}$}
\vcenter{\hbox{$#2#3$}}\kern-.5\wd0}}

\newcommand\bE{\mathbb{E}}
\newcommand\bG{\mathbb{G}}
\newcommand\bH{\mathbb{H}}
\newcommand\bL{\mathbb{L}}

\newcommand\bN{\mathbb{N}}

\newcommand\bR{\mathbb{R}}

\newcommand\bT{\mathbb{T}}

\newcommand\bZ{\mathbb{Z}}

\newcommand\fR{\mathbf{R}}

\newcommand\cA{\mathcal{A}}
\newcommand\cB{\mathcal{B}}

\newcommand\cD{\mathcal{D}}
\newcommand\cF{\mathcal{F}}
\newcommand\cG{\mathcal{G}}

\newcommand\cL{\mathcal{L}}
\newcommand\cP{\mathcal{P}}

\newcommand\cM{\mathcal{M}}
\newcommand\cT{\mathcal{T}}
\newcommand\cO{\mathcal{O}}

\newcommand\rB{\mathscr{B}}
\newcommand\rF{\mathscr{F}}
\newcommand\rG{\mathscr{G}}

\newcommand\aint{-\hspace{-0.38cm}\int}

\newcommand{\mysection}[1]{\section{#1}
\setcounter{equation}{0}}

\begin{document}

\title[A maximal regularity theory of moments]
{An $L_p$-maximal regularity estimate of moments  of  solutions to second-order stochastic partial differential equations}

\author{Ildoo Kim}
\address{Department of mathematics, Korea university, 1 anam-dong
sungbuk-gu, Seoul, south Korea 136-701}
\email{waldoo@korea.ac.kr}
\thanks{The author has been supported by the National Research Foundation of Korea grant funded by the Korea government (NRF-2020R1A2C1A01003959)}

\subjclass[2010]{60H15, 35R60}

\keywords{Maximal regularity moment estimate, Stochastic partial differential equations}

\begin{abstract}
We obtain   uniqueness and existence of a solution $u$ to the following second-order stochastic partial differential equation (SPDE) :
\begin{align}
								\label{abs eqn}
				 du= \left( \bar a^{ij}(\omega,t)u_{x^ix^j}+ f  \right)dt + g^k dw^k_t, \quad t \in (0,T); \quad u(0,\cdot)=0,
\end{align}
where $T \in (0,\infty)$, $w^k$ $(k=1,2,\ldots)$ are independent Wiener processes, $(\bar a^{ij}(\omega,t))$ is  a (predictable) nonnegative symmetric matrix valued stochastic process such that
$$
\kappa |\xi|^2 \leq \bar a^{ij}(\omega,t) \xi^i \xi^j  \leq  K |\xi|^2  \qquad \forall (\omega,t,\xi) \in \Omega \times (0,T) \times {\mathbf{R}}^d
$$
for some $\kappa, K \in (0,\infty)$,
$$
f \in L_p\left( (0,T) \times {\mathbf{R}}^d,  dt \times dx  ; L_r(\Omega, {\mathscr{F}} ,dP) \right),
$$
and 
$$
g, g_x \in L_p\left( (0,T) \times {\mathbf{R}}^d,  dt \times dx  ; L_r(\Omega, {\mathscr{F}} ,dP; l_2) \right)
$$
with $2 \leq r \leq p < \infty$ and appropriate measurable conditions. 
Moreover, for the solution $u$, we obtain the following maximal  regularity moment estimate 
\begin{align}
									\notag
&\int_0^T \int_{{\mathbf{R}}^d}\left(\bE\left[|u(t,x)|^r\right] \right)^{p/r} dx dt + \int_0^T \int_{{\mathbf{R}}^d}\left(\bE\left[|u_{xx}(t,x)|^r\right] \right)^{p/r} dx dt \\
										\notag
&\leq N \bigg(\int_0^T \int_{{\mathbf{R}}^d}\left(\bE\left[|f(t,x)|^r\right] \right)^{p/r} dx dt + \int_0^T \int_{{\mathbf{R}}^d}\left(\bE\left[|g(t,x)|_{l_2}^r\right] \right)^{p/r} dx dt  \\
										\label{abs est}
&\qquad + \int_0^T \int_{{\mathbf{R}}^d}\left(\bE\left[|g_x(t,x)|_{l_2}^r\right] \right)^{p/r} dx dt \bigg),
\end{align}
where $N$ is a positive constant depending only on $d$, $p$, $r$, $\kappa$, $K$, and $T$.
As an application, for the solution $u$ to \eqref{abs eqn}, the $r$-th moment $m^r(t,x):=\bE|u(t,x)|^r$ is in the parabolic Sobolev space $W_{p/r}^{1,2}\left((0,T) \times \mathbf{R}^d\right)$. 
\end{abstract}

\maketitle

\mysection{introduction}

Studying the second-order stochastic partial differential equations  of the form of \eqref{abs eqn} have been regarded as important problems in both mathematics and engineering communities for a long time since these equations naturally appear in filtering problems (cf. \cite{rozovsky2018stochastic,bain2008fundamentals}). 

Since Krylov  (\cite{krylov1996l_p}) initiated an $L_p$-theory of stochastic partial differential equations  \eqref{abs eqn}, there have been lots of articles developing this theory in various directions. For instance, here are some examples of generalizations  with respect to coefficients (\cite{auscher2014conical,kim2009sobolev,krylov2009divergence,porver2019}), domains (\cite{krylov1999sobolev,kim2004stochastic,kim2014weighted}), operators (\cite{mikulevicius2019cauchy,van2012stochastic,kim2013parabolic}), nonlinear equations (\cite{zhang2006lp,zhang2007regularities}), fractional derivatives (\cite{kim2016lp,kim2019sobolev}), general noises (\cite{chen2014lp,kim2012lp}), degenerate equations (\cite{gerencser2015solvability,kim2019sharp}), and systems (\cite{kim2013note,mikulevicius2001note}).

Surprisingly, to the best of our knowledge, there is no result considering  generalization of optimal exponents of moments for solutions and free data even though finite moment conditions play very important roles even in classical probability theories such as central limit theorems, law of large numbers, and etc (cf. \cite{durrett2019probability}).

If $p=r$, then \eqref{abs est} is obtained by Krylov (\cite{Krylov1999}). For $r<p$, we believe that this paper is the first attempt to obtain the maximal regularity estimate of the type \eqref{abs est}.

It seems to be natural that the moments of solutions are more regular even though solutions are comparatively irregular due to the effect of random noises.
Indeed, as an easy application of our theory, we show that the $r$-th moment of the solution $m^r(t,x):=\bE|u(t,x)|^r$ is in the parabolic Sobolev space $W_{p/r}^{1,2}\left((0,T) \times \mathbf{R}^d\right)$ for $2 \leq r \leq p <\infty$ (see Definition \ref{parabolic sobolev}).

We need to mention that our estimates are non-trivial even without stochastic noises. Usually, if there is no stochastic term $g$ in \eqref{abs eqn}, estimates easily come from deterministic theories by solving equations for each random parameter $\omega$. 
However, our estimate \eqref{abs est} cannot be obtained directly from deterministic estimates even though $g=0$ (see Theorem \ref{thm 3-2} below).

To obtain our estimate, we use classical harmonic analysis tools such as heat kernel estimates, the Hardy-Littlewood maximal theorem, the Fefferman-Stein theorem, and the Marcinkiewicz interpolation theorem.
These classical tools fit to quasilinear operators. However, since there appear some non-quasilinear operators in order to obtain our estimates, we develop a variant of Marcinkiewicz's interpolation theorem which could be applied in our setting for non-quasilinear operators (see Lemma \ref{ext lem} below).

This paper is organized as follows. In Section 2 we introduce stochastic Banach spaces and our main theorem, Theorem \ref{main thm}.  Heat kernel estimates are given  in Section 3.
A maximal regularity moment  estimate for solutions without random noises is given in Section 4. Boundedness of integral operators related to solution representations is given in Section 5. Finally, we prove all our main theorems in Section 6. 

We finish the introduction with  notation used in the article. 
\begin{itemize}

\item We use Einstein's summation convention throughout this paper. 

\item $\bN$ and $\bZ$ denote the natural number system and the integer number system, respectively.
$\bZ_+$ is the set of all nonnegative integers, i.e. $\bZ_+ =\{0,1,2,\ldots\}$.  
As usual $\fR^{d}$
stands for the Euclidean space of points $x=(x^{1},...,x^{d})$.
 For $i=1,...,d$, multi-index $\alpha=(\alpha_{1},...,\alpha_{d})$,
$\alpha_{i}\in\{0,1,2,...\}$, and functions $u(x)$ we set
$$
u_{x^{i}}=\frac{\partial u}{\partial x^{i}}=D_{i}u,\quad
D^{\alpha}u=D_{1}^{\alpha_{1}}\cdot...\cdot D^{\alpha_{d}}_{d}u.
$$
For $\alpha_i =0$, we define $D^{\alpha_i}_i u = u$. 
Simply, we use $u_x$ (or $\partial_x u$) and $u_{xx}$ (or $\partial_{xx}u$) to denote the gradient vector of $u$ and the Hessian matrix of $u$, respectively. 
%

\item $C^\infty(\fR^d)$ denotes the space of infinitely differentiable functions on $\fR^d$. 
By $C_c^\infty(\fR^d)$,  we denote the subspace of $C^\infty(\fR^d)$ with the compact support.

\item For $p \in (0,\infty)$, a quasi normed space $F$,
and a  measure space $(X,\mathcal{M},\mu)$, 
by $L_{p}(X,\cM,\mu;F)$,
we denote the space of all $F$-valued $\mathcal{M}^{\mu}$-measurable functions
$u$ so that
\[
\left\Vert u\right\Vert _{L_{p}(X,\cM,\mu;F)}:=\left(\int_{X}\left\Vert u(x)\right\Vert _{F}^{p}\mu(dx)\right)^{1/p}<\infty,
\]
where $\mathcal{M}^{\mu}$ denotes the completion of $\cM$ with respect to the measure $\mu$. 
For $u \in L_{p}(X,\cM,\mu;F)$, we say that $v$ is a {\bf modification} of $u$ if
$$
u(x)=v(x) \qquad (\mu \text{-} a.e.),
$$ 
that is, there exists a subset $X_0 \subset X$ such that 
$$
u(x) =v(x) \qquad \forall x \in X_0
$$
and $\mu (X \setminus X_0) =0$. 
For $p=\infty$, we write $u \in L_{\infty}(X,\cM,\mu;F)$ if
$$
 \|u\|_{L_{\infty}(X,\cM,\mu;F)} 
:= \inf\left\{ \nu \geq 0 : \mu( \{ x: \|u(x)\|_F > \nu\})=0\right\} <\infty,
$$
where $\sup$ denotes the essential supremum with respect to measure $\mu$.
If there is no confusion for the given measure and $\sigma$-algebra, we usually omit them.
Moreover, if $F=\fR$, then $F$ will be omitted. Especially, $L_p$ denotes the $L_{p}(X,\cM,\mu;F)$ space, where $X=\fR^d$, $\cM$ is the Lebesgue measurable sets, $\mu$ is the Lebesgue measure, and $F=\fR$. 

\item For  a Lebesgue measurable set  $\cO \subset \fR^d$, $|\cO|$ denotes the Lebesgue measure of $\cO$.


\item By $\cF$ and $\cF^{-1}$ we denote the d-dimensional Fourier transform and the inverse Fourier transform, respectively. That is,
$\cF[f](\xi) := \int_{\fR^{d}} e^{-i x \cdot \xi} f(x) dx$ and $\cF^{-1}[f](x) := \frac{1}{(2\pi)^d}\int_{\fR^{d}} e^{ i\xi \cdot x} f(\xi) d\xi$.

\item If we write $N=N(a,b,\cdots)$, this means that the
constant $N$ depends only on $a,b,\cdots$. 
\end{itemize}




\mysection{Setting and main results}

Let $(\Omega,\rF,P)$ be a complete probability space,
$\{\rF_{t},t\geq0\}$ be an increasing filtration of
$\sigma$-fields $\rF_{t}\subset\rF$, each of which contains all
$(\rF,P)$-null sets. By  $\cP$ we denote the predictable
$\sigma$-algebra generated by $\{\rF_{t},t\geq0\}$ and  we assume that
on $\Omega$ there exist  independent one-dimensional Wiener
processes $w^{1}_{t},w^{2}_{t},...$, each of which is a Wiener
process relative to $\{\rF_{t},t\geq0\}$. 

We study the following initial value problem throughout the paper:
\begin{align}
				\label{main eqn}
				 du= \left( \bar a^{ij}(t)u_{x^ix^j}+ f  \right)dt + g^k dw^k_t, \quad t \in (0,T); \quad u(0,\cdot)=0. 
\end{align}
As mentioned in the introduction,  Einstein's summation convention with respect to indices $i,j,k$ is assumed and  the argument for random parameter $\omega \in \Omega$ is omitted (mostly)  in the above equation for the simplicity of notation. We fix $d \in \bN$ to denote the space dimension throughout the paper. 

First, we introduce some deterministic function spaces related to our results. For $p \in (0,\infty]$ and a nonnegative integer $\gamma$, let
$H_{p}^{\gamma}=H_{p}^{\gamma}(\fR^{d})$ denote  the class of all
(tempered) distributions $u$  on $\fR^{d}$ such that
\begin{equation*}
D^{\alpha}_x u\in L_p(\fR^d), \, \,\,|\alpha|\leq \gamma.
\end{equation*}
Here $H_p^\gamma$ is called the Sobolev space with the order $\gamma$ and the exponent $p$. 
Let  $l_2$ denote the set of all sequences $a=(a^1,a^2,\cdots)$ such that
$$
|a|_{l_{2}}:= \left(\sum_{k=1}^{\infty}|a^{k}|^{2}\right)^{1/2}<\infty.
$$
Similarly, for $p  \in (0,\infty]$ and a nonnegative integer $\gamma$,
$H_{p}^{\gamma}(l_2)=H_{p}^{\gamma}(\fR^{d} ; l_2)$ denote  the class of all
$l_2$-valued (tempered) distributions $u$  on $\fR^{d}$ such that
\begin{equation*}
D^{\alpha}_x u\in L_p(\fR^d;l_2), \, \,\,|\alpha|\leq \gamma.
\end{equation*}
 It is well-known that for $p \in [1,\infty]$, $H_p^\gamma$ and $H_p^{\gamma}(l_2)$ are  Banach spaces with the norms
 $$
\|u\|_{H_p^\gamma} :=  \sum_{ |\alpha| \leq \gamma} \| D^\alpha u\|_{L_p(\fR^d)}
 $$
 and
  $$
\|u\|_{H_p^\gamma(l_2)} :=  \sum_{ |\alpha| \leq \gamma} \| |D^\alpha u|_{l_2}\|_{L_p(\fR^d)},
 $$
 respectively (see \cite{Krylov1999,Krylov2008} and references therein). 
Set $ L_p=H^0_p$ and  $L_p(l_2) = H^0_p(l_2)$.
Next, we introduce stochastic spaces.  
For $T \in (0,\infty]$, $p \in (0,\infty]$, and a nonnegative integer $\gamma$, we denote
$$
 \bH^{\gamma}_{p}(T):=L_p(\Omega\times (0,T),  \rF \times \cB\left((0,T)\right), P\times dt ; H^{\gamma}_p),
$$
and its predictable subspace $\tilde \bH^{\gamma}_{p}(T)$, i.e. $u \in  \tilde \bH^{\gamma}_{p}(T)$ if there exists a $\cP \times \cB(\fR^d)$-measurable modification $\tilde u$ of $u$ such that
$$
\tilde u \in  \bH^{\gamma}_{p}(T),
$$
where $\cB\left( (0,T) \right)$  and $\cB(\fR^d)$ are the Borel sets of $(0,T)$  and $\fR^d$, respectively and $dt$  denotes the Lebesgue measures on $(0,T)$.
In other words,
$$
\tilde \bH^{\gamma}_{p}(T) =L_p(\Omega\times (0,T),  \cP, P\times dt ; H^{\gamma}_p).
$$
Similarly, we define
$$
\bH^{\gamma}_{p}(T,l_2):=L_p(\Omega\times (0,T),  \rF \times \cB\left((0,T)\right), P\times dt ; H^{\gamma}_p(l_2))
$$
and its predictable subspace $\tilde \bH^{\gamma}_{p}(T,l_2)$.
For the notational convenience, we set
$$
\bL_p(T) := \bH_p^0(T), \quad \text{and} \quad \bL_p(T,l_2):= \bH_p^0(T,l_2).
$$
Moreover, for $p,r\in (0,\infty]$, we define
$\bL_{p,r}(T)$ and $\tilde \bL_{p,r}(T)$ by the $\rF \times \cB((0,T)) \times \cB(\fR^d)$-measurable and the $\cP \times \cB(\fR^d)$-measurable
subspaces of
$$
L_p\left((0,T) \times \fR^d,  \cB\left( (0,T) \right) \times \cB\left(\fR^d\right) ; L_r(\Omega,\rF)  \right),
$$
respectively,
 i.e. $u \in \bL_{p,r}(T)$ if there exists a  $\rF \times \cB((0,T))  \times  \cB(\fR^d)$- measurable modification  $\tilde u(\omega,t,x)$ of $u$ on $\Omega \times (0,T) \times \fR^d$  such that
$$
\|u\|_{\bL_{p,r}(T) } := \left(\int_0^T \int_{\fR^d}  \left\| \tilde u(\cdot,t,x) \right\|_{L_r(\Omega)}^{p} dx dt \right)^{1/p} < \infty
$$
and $u \in \tilde \bL_{p,r}(T)$ if there exists a $\cP \times \cB(\fR^d)$-measurable modification $\tilde u$ of $u$ such that
$$
\left(\int_0^T \int_{\fR^d}  \left\| \tilde u(\cdot,t,x) \right\|_{L_r(\Omega)}^{p} dx dt \right)^{1/p} < \infty. 
$$
Note that for any bounded subset $U \subset \bR$ and $u \in \bL_{p,r}(T)$, applying H\"older's inequality and Minkowski's inequality 
\begin{align}
							\notag
\left(\bE\left[ \left|\int_0^T \int_{U} |u(t,x)| dx dt\right|^r \right]\right)^{1/r}
&\leq\int_0^T \int_{U}  \left(\bE\left[  |u(t,x)|^r \right]\right)^{1/r} dx dt \\
							\label{2020111501}
&\leq N\|u\|_{\bL_{p,r}(T)} < \infty. 
\end{align}
Thus for any $u \in \bL_{p,r}(T)$, $u(t,x)$ is locally integrable function on $(0,T) \times \bR^d$ almost surely and it implies that 
weak derivatives of $u$ with respect to $t$ and $x$ exist almost surely. 
Finally for any positive integer $\gamma$,  we write
$$
u \in \bH^\gamma_{p,r}(T)
$$
and
$$
\|u \|_{\bH^\gamma_{p,r}(T)} :=\|u\|_{\bL_{p,r}(T)} + \sum_{|\alpha|=\gamma} \|D^\alpha u \|_{\bL_{p,r}(T)}
$$
if $u \in \bL_{p,r}(T)$ and $D_x^\alpha u \in \bL_{p,r}(T)$ for all multi indexes $|\alpha| = \gamma$.
The notation $\tilde \bH^\gamma_{p,r}(T)$ is used to denote the predictable subspace of $\bH^\gamma_{p,r}(T)$, i.e. 
$u  \in \tilde \bH^\gamma_{p,r}(T)$ if $u \in \tilde \bL_{p,r}(T)$ and $D_x^\alpha u \in \tilde \bL_{p,r}(T)$ for all multi indexes $|\alpha| = \gamma$. Similarly, the spaces $\bH^\gamma_{p,r}(T,l_2)$ and  $\tilde \bH^\gamma_{p,r}(T,l_2)$ can be defined for all $p,r\in (0,\infty]$ and $\gamma \in \bZ_+$.

\begin{remark}
					\label{re banach}
\begin{enumerate}[(i)]
\item If $p \in (0,1)$ or $r \in (0,1)$, then the given spaces $ \bH^\gamma_{p}(T)$, $ \bH^\gamma_{p}(T,l_2)$, $ \bH^\gamma_{p,r}(T)$,  and $ \bH^\gamma_{p,r}(T,l_2)$ are not Banach spaces but complete metric spaces. 

\item The predictable subspaces $\tilde \bH^\gamma_{p}(T)$, $\tilde \bH^\gamma_{p}(T,l_2)$, $\tilde \bH^\gamma_{p,r}(T)$,  and $\tilde \bH^\gamma_{p,r}(T,l_2)$ are closed
subspaces of $ \bH^\gamma_{p}(T)$, $\bH^\gamma_{p}(T,l_2)$, $ \bH^\gamma_{p,r}(T)$, and $\bH^\gamma_{p,r}(T,l_2)$, respectively.
Moreover, if $p,r \in (1,\infty)$, then  $\tilde \bH^\gamma_{p}(T)$, $\tilde \bH^\gamma_{p}(T,l_2)$, $\tilde \bH^\gamma_{p,r}(T)$, $\tilde \bH^\gamma_{p,r}(T,l_2)$ are reflexive Banach spaces (cf. \cite[Theorem 1.3.10 and Theorem 1.3.21]{hytonen2016analysis}).
In particular, any linear bounded functional $\bL$ on $\tilde \bL_{p,r}(T,l_2)$ with  $T \in (0,\infty]$, and $p,r \in (1,\infty)$ is given by
$$
\bL(u) =   \int_0^T \int_{\fR^d} \bE \left[ \sum_{k=1}^\infty \left(u^k(\omega,t,x) v^k(\omega,t,x) \right)  \right] dx dt
$$
for some $v \in \tilde \bL_{p^\ast, r^\ast}(T,l_2)$ with $p^\ast= p/(p-1)$ and $r^\ast = r/(r-1)$. 
\end{enumerate}
\end{remark}

Next we introduce the definition of solutions to the stochastic equations.
Let $\cD$ be the space of all distributions (generalized functions) on $C_c^\infty(\fR^d)$, 
and let $\cD(l_2)$ denote the space of  all $l_2$-valued distributions (generalized functions) on $C_c^\infty(\fR^d)$.

\begin{definition}
					\label{def sol-2}
Let $u_0$ be a $\cD$-valued random variable, $u$ and $F$ be a $\cD$-valued stochastic process, and $g$ be a $\cD(l_2)$-valued predictable stochastic process. 
We say that  a $\cD$-valued stochastic process $u$ satisfies (or is a solution to) the equation
\begin{align*}
&du(t,x)= F(t,x) dt +g(t,x) dw^k_t, \quad (t,x) \in (0,T)  \times \fR^d \\
&u(0,\cdot)=u_0
\end{align*}
in the sense of distributions if   for any $\phi\in C^{\infty}_0(\fR^d)$, the equality that
\begin{align*}
(u(t,\cdot),\phi)=   (u_0 ,\phi) + \int^t_0 (F(s,\cdot),\phi)ds+\sum_k \int^t_0 (g^k(s,\cdot),\phi)dw^k_t
\end{align*}
holds for all $t < T$ $(a.s.)$.  

In particular,   we say that  $u \in \bH_{p,r}^2(T)$ is a  solution to \eqref{main eqn} if for any $\phi\in C^{\infty}_0(\fR^d)$,
\begin{eqnarray}
(u(t,\cdot),\phi)
											\label{def sol-2}
&=&   (u_0 ,\phi) + \int^t_0  \left[ \left(  a^{ij}(t) u(s,\cdot), \phi_{x^ix^j} \right) + (f(s,\cdot),\phi) \right]ds \\
											\notag
&& +\sum_k \int^t_0 \left( g(s,\cdot),\phi \right) dw^k_t \qquad \forall t \in (0,T)~ (a.s.).
\end{eqnarray}
It is easy to check that \eqref{def sol-2} makes sense for all $u \in \bH_{p,r}^2(T)$, $f \in \bL_{p,r}(T)$, and $g \in \tilde \bH^1_{p,r}(T,l_2)$ due to \eqref{2020111501}.
\end{definition}

\begin{theorem}
				\label{main thm}
Let $2 \leq r \leq p <\infty$ and $T \in (0,\infty)$.
Assume that the coefficients $\bar a^{ij}(\omega,t)$ is nonnegative symmetric matrix-valued predictable and satisfy the following ellipticity condition
\begin{align*}
\kappa |\xi|^2 \leq \bar a^{ij}(\omega, t) \xi^i \xi^j \leq K |\xi|^2 \qquad  \forall (\omega, t,\xi) \in \Omega \times [0,\infty) \times \fR^d.
\end{align*}
Then for all $f \in \bL_{p,r}(T)$ and $g \in \tilde \bH^1_{p,r}(T,l_2)$, there exists a unique solution $u \in \bH_{p,r}^2(T)$ to equation \eqref{main eqn} such that
\begin{align}
								\label{main est}
\|u\|_{\bH_{p,r}^2(T)} \leq N_1\left( \|f\|_{\bL_{p,r}(T)} + \|g\|_{\bH^1_{p,r}(T,l_2)}\right),
\end{align}
and
\begin{align}
								\label{main est-2}
\|u_{xx}\|_{\bL_{p,r}(T)} \leq N_2\left( \|f\|_{\bL_{p,r}(T)} + \|g_x\|_{\bL_{p,r}(T)}\right),
\end{align}
where $N_1=N_1(d,p,r,\kappa,K,T)$ and $N_2=N_2(d,p,r,\kappa,K)$. 
\end{theorem}

\begin{remark}
\begin{enumerate}[(i)]
\item If the deterministic free data $f \in \tilde \bL_{p,r}(T)$, then the solution $u$ is also predictable, i.e. $u \in  \tilde \bH_{p,r}^2(T)$.

\item We considered zero initial value problem for simplicity. 
Combining with classical results (cf. \cite{Krylov1999}), we can handle the solvability of the following initial value problem 
$$
				 du= \left( \bar a^{ij}(\omega,t)u_{x^ix^j}+ f  \right)dt + g^k dw^k_t, \quad t \in (0,T); \quad u(0,\cdot)=u_0,
$$
where $u_0$ is in a Sobolev space-valued or Besov space-valued $L_p(\Omega)$-space.


\item  We can construct  similar initial value problems with stopping times $\tau \leq T$ instead of the deterministic time $T$ by considering the following extension of the free data
$$
1_{ t \leq \tau} f \quad \text{and} \quad 1_{t \leq \tau} g.
$$

\item If we only want to control the $\bL_{p,r}(T)$-norm of the solution $u$ by $f$ and $g$, 
then the range of $r$ and $p$ can be relaxed to $p,r \in [2,\infty)$, that is, the condition $r \leq p$ is not necessary (see Lemma \ref{lem20200110} and Lemma \ref{lem2020011010} below). 
In other words,  for $p,r \in [2,\infty)$,
\begin{align*}
\|u\|_{\bL_{p,r}(T)} \leq N(d,p,r,\kappa,K,T)\left( \|f\|_{\bL_{p,r}(T)} + \|g\|_{\bL_{p,r}(T,l_2)}\right).
\end{align*}

\item Assume that $g=0$ in \eqref{main eqn}. Then the range of $r$ and $p$ can be extended to $1 < r \leq p  < \infty$ 
and the predictable assumption on $\bar a^{ij}(t)$ can be weaken to the $\rF \times \cB((0,\infty))$-measurable condition (see Theorem \ref{thm 3-2} below). 
Similarly to (iv), the restriction that $r \leq p$ is not required as well if we merely control the $\bL_{p,r}(T)$-norm of the solution $u$.
More precisely, for $p,r \in [1,\infty]$ (Lemma \ref{lem20200110} below),
\begin{align*}
\|u\|_{\bL_{p,r}(T)} \leq N(d,p,r,\kappa,K,T)\left( \|f\|_{\bL_{p,r}(T)} \right).
\end{align*}
Moreover, if we  choose a data $f$ in a slightly different space, i.e. 
$$
f \in L_p\left(\Omega, \rF ,  P; L_r\left( (0,T), \cB((0,T)) ;L_r \right) \right),
$$
then the range of $p$ and $r$ can be relaxed to $p \in (0,\infty]$ and $r \in (1,\infty)$ (see Theorem \ref{thm 3-1} below). 
\end{enumerate}
\end{remark}

To state our second main theorem, we need a parabolic version of Sobolev space.
\begin{definition}[parabolic Sobolev space]
								\label{parabolic sobolev}
By $W_p^{1,2}\left( (0,T) \times \fR^d \right)$ with $p \in (1,\infty)$, we denote  the space of $\cB\left((0,T) \times \fR^d \right)$-measurable functions $m(t,x)$ such that
$$
\int_0^T \int_{\fR^d} |m(t,x)|^p dx dt < \infty,
$$
$$
\int_0^T \int_{\fR^d} |\partial_t m(t,x)|^p dx dt < \infty,
$$
$$
\int_0^T \int_{\fR^d} |\partial_{xx}m(t,x)|^p dx dt 
:=\int_0^T \int_{\fR^d} \sum_{i,j} |m_{x^ix^j}(t,x)|^p dx dt< \infty.
$$
\end{definition}
\begin{theorem}
				\label{main thm 2}
Assume that the coefficients $\bar a^{ij}(\omega,t)$ is nonnegative symmetric matrix-valued predictable and satisfies the following ellipticity condition
\begin{align*}
\kappa |\xi|^2 \leq \bar a^{ij}(\omega, t) \xi^i \xi^j \leq K |\xi|^2 \qquad  \forall (\omega, t,\xi) \in \Omega \times [0,\infty) \times \fR^d.
\end{align*}
Let $2 \leq r \leq p <\infty$, $T \in (0,\infty)$, $f \in \bL_{p,r}(T)$ and $g \in \tilde \bH^1_{p,r}(T,l_2)$, and $u$ be the solution to \eqref{main eqn} in $\bH_{p,r}^2(T)$.
Then the $r$-th moment of the solution $u$ is in $W_{p/r}^{1,2}\left((0,T) \times \fR^d \right)$.
More precisely, defining $m^r(t,x):=\bE|u(t,x)|^r$, we have
\begin{align}
								\notag
&\int_0^T \int_{\fR^d} |m^r(t,x)|^{p/r} dx dt
+\int_0^T \int_{\fR^d} |\partial_t m^r(t,x)|^{p/r} dx dt
+\int_0^T \int_{\fR^d} |\partial_{x x}m^r(t,x)|^{p/r} dx dt \\
								\label{main est 2}
&\leq N\left( \|f\|^p_{\bL_{p,r}(T)} + \|g\|^p_{\bH^1_{p,r}(T)}\right),
\end{align}
where $N=N(d,p,r,\kappa,K,T)$. 
\end{theorem}

\mysection{Kernel Estimates}

In this section, we prove kernel estimates related to deterministic coefficients $a^{ij}(t)$. These estimates might be well-known. 
However, we could not find an exact reference which fit to our setting. 
Moreover, if we consider random coefficients, then two different types of  H\"ormander's condition appear (see Remark \ref{rmk 20200129} below).
Thus we give the details of kernel estimates for the completeness of the paper.

Assume that there exists a  $d \times d$ nonnegative symmetric matrix-valued (non-random) function $a(t)=\left(a^{ij}(t) \right)$ on $[0, \infty) $ satisfies the following ellipticity condition
\begin{align*}
\kappa |\xi|^2 \leq a^{ij}(t) \leq K |\xi|^2 \qquad  \forall (t,\xi) \in [0,\infty) \times \fR^d.
\end{align*}
For $t>\rho$, we define the following $d \times d$ matrices
$$
A_{ t \rho}= \int_\rho^t a(\eta) d\eta, \quad B_{t \rho}= A^{-1}_{t \rho}, \quad \text{and} \quad \sigma_{t \rho}=A^{1/2}_{t \rho}.
$$
It is easy to check that
\begin{align*}
\kappa (t-\rho) |\xi|^2 \leq A_{t \rho}^{ij} \xi^i \xi^j \leq K (t-\rho)|\xi|^2 \qquad  \forall t>\rho, \quad \forall  \xi \in  \fR^d,
\end{align*}
\begin{align*}
K^{-1} (t-\rho)^{-1} |\xi|^2 \leq B_{t\rho}^{ij} \xi^i \xi^j \leq \kappa^{-1} (t-\rho)^{-1} |\xi|^2 \qquad  \forall t>\rho, \quad \forall  \xi \in  \fR^d,
\end{align*}
and
\begin{align}
					\label{2019081601}
\kappa (t-\rho)^{1/2} |\xi|^2 \leq \sigma_{t\rho}^{ij} \xi^i \xi^j \leq K (t-\rho)^{1/2} |\xi|^2 \qquad  \forall t>\rho, \quad \forall  \xi \in  \fR^d.
\end{align}

Define
\begin{align}
					\label{ker def}
p(t,\rho,x) := 1_{\rho<t} (4\pi)^{-d/2} (\det B_{t \rho} )^{1/2} \exp ( - (B_{t \rho} x, x)/4 ).
\end{align}
Then
$$
p(t,\rho,x) = 1_{\rho<t} (4\pi)^{-d/2} (\det \sigma_{t \rho})^{-1} \exp ( - ( |\sigma_{t\rho}^{-1} x|^2/4 )
$$
and by \eqref{2019081601}
\begin{align*}
\kappa^d (t-\rho)^{d/2} \leq |\det \sigma_{t \rho} | \leq K^d (t-\rho)^{d/2}
\end{align*}
Thus for any $d$-dimensional multi-index $\gamma$, there exist positive constants $N$ and $c_0$ such that
\begin{align}
							\label{ker est 1}
\left|D_x^\gamma p(t,\rho,x)  \right| \leq N \left(  (t-\rho)^{-d/2 -\frac{\gamma}{2}} \right) \exp\left( -c_0 (t-\rho)^{-1}|x|^2\right) \qquad \forall t> \rho, ~ x \in \fR^d,
\end{align}
where $N=N(d,\gamma , \kappa, K)$ and $c_0=c_0(d,\gamma, \kappa)$.

Moreover, by the Fourier transform, 
$$
\cF \left[ D^\gamma p(t,\rho,\cdot) \right] (\xi) =  i^{|\gamma|}\xi^\gamma \exp \left(  -A^{ij}_{t \rho} \xi^i \xi^j \right)
$$
and
\begin{align*}
\cF \left[ D_tD^\gamma p(t,\rho,\cdot) \right] (\xi) 
&=  i^{|\gamma|}\xi^\gamma a(t)^{ij} \xi^i \xi^j \exp \left(  -A^{ij}_{t \rho} \xi^i \xi^j \right)  \\
&= a^{ij}(t) \cF \left[ D^\gamma p_{x^ix^j}(t,\rho,\cdot) \right] (\xi).
\end{align*}
Therefore
\begin{align}
						\label{ker est 2}
\left| D_t D_x^\gamma p(t,\rho,x)  \right| \leq N \left(  (t-\rho)^{-\frac{d}{2} -\frac{\gamma}{2} -1} \right) \exp\left( -c_0 (t-\rho)^{-1}|x|^2\right) \qquad \forall t> \rho, ~ x \in \fR^d.
\end{align}

Next we introduce cylinders in $\fR^{d+1}$ and obtain H\"ormander's type kernel estimates for $p(t,\rho,x)$ based on these cylinders. 
For $c>0$ and $(t_0,x_0 ) \in \fR^{d+1}$, we define the cylinder whose center is $(t_0,x_0)$ and radius is $c$ in $\fR^{d+1}$ as 
\begin{align*}
Q_c(t_0,x_0) 
&:= \left(t_0 - c^2, t_0+ c^2\right) \times B_c(x_0) \\
&:= \left(t_0 - c^2, t_0+ c^2\right) \times  \{ x \in \fR^d : |x-x_0| < c\}.
\end{align*}
Note that the kernel $p(t,\rho,x)$ is defined on  $\fR^{d+2}$ due to the indicator $1_{ \rho < t }$, i.e. 
$$
p(t,\rho,x) := (4\pi)^{-d/2} (\det B_{t \rho} )^{1/2} \exp ( - (B_{t \rho} x, x)/4 ) \quad \text{if} \quad \rho < t
$$
and
$$
p(t,\rho,x) := 0 \quad \text{if} \quad t \leq \rho.
$$
\begin{lemma}
					\label{hor lem}
Let  $\gamma$ be a $d$-dimensional multi-index, $ r \in [1,\infty)$, $c \in (0,\infty)$, $(t_0,x_0) \in \fR^{d+1}$,  $(t,x), (s,y) \in Q_c(t_0,x_0)$, and
$$
\cA(\rho,t_0,x_0) := \left\{ z \in \fR^{d} :  |\rho -t_0|^{1/2}   + |z-x_0| \geq 8c  \right\}.
$$
Assume
\begin{align}
						\label{2019110301}
r |\gamma|=2.
\end{align}
Then there exists a positive constant $N=N(d, \kappa, K, r, |\gamma|)$ such that 
\begin{align*}
\int_{  -\infty }^\infty  \left[ \int_{A(\rho,t_0,x_0)} | (D_x^\gamma p)(t,\rho,x-z) -  (D_x^\gamma p)(s,\rho,y-z)| dz \right]^{r} d\rho 
\leq N.
\end{align*}
\end{lemma}

\begin{proof}
We may assume that $t \geq s$ without loss of generality.
Moreover we only prove the case $t>s$ since the proof of the case $t=s$ is simpler. 

Set
\begin{align}
					\notag
I_1(t_0,c)
&:= \left\{ \rho \in \fR :  |\rho -t_0|  > (4c)^2  \right\} \\
								\label{2019110101}
&\subset 
 \left\{ \rho \in \fR :  |\rho -t|  > 4c^2  \right\}
 \bigcap  \left\{ \rho \in \fR :  |\rho -s|  > 4c^2  \right\}
\end{align}
and
\begin{align*}
I_2(t_0,c)
&:= \left\{ \rho \in \fR  :  |\rho -t_0| \leq (4c)^2  \right\} .
\end{align*}
Then obviously,
$$
\fR = I_1(t_0,c) \cup I_2(t_0,c)
$$
and for all $\rho \in I_2(t_0,c)$, 

\begin{align}
						\notag
&\cA(\rho, t_0, x_0)  \subset    \left\{ z \in \fR^{d} :  |z-x_0| \geq 4c  \right\} \\
									\label{2019110102}
&\subset   \bigcap_{ \theta \in [0,1] } \left\{ z \in \fR^{d} :  |z- (\theta x + (1-\theta)y) | \geq 2c  \right\}  .
\end{align}
By the fundamental theorem of calculus and the equalities
$$
1_{\rho < t} = 1_{\rho <t} 1_{\rho < s } + 1_{\rho <t}1_{\rho \geq s}
= 1_{\rho < s }  + 1_{ s \leq \rho < t},
$$
we have 
\begin{align*}
&\left|(D^\gamma_x p)(t,\rho,x-z) -   (D^\gamma_xp)(s,\rho,y-z) \right| \\
&=\left| 1_{\rho<t} (D^\gamma_x p)(t,\rho,x-z) -  1_{\rho<s} (D^\gamma_xp)(s,\rho,y-z) \right| \\
&=|1_{\rho<s} (D^\gamma_x p) (s,\rho,y-z) -  1_{\rho<s} (D^\gamma_xp) (t,\rho,y-z) + 1_{ s \leq \rho<t } (D^\gamma_x p)(t,\rho,y-z)  \\
&\qquad +1_{\rho<t} (D_x^\gamma p)(t,\rho,y-z) -  1_{\rho<t} (D^\gamma_x p)(t,\rho,x-z) |\\
&=\bigg|1_{\rho<s} \int_0^1  \frac{d}{d\theta} \left( (D^\gamma_x p)( \theta t + (1-\theta)s,\rho,y-z)  \right)d \theta + 1_{s \leq \rho <t} (D^\gamma_x p)(t,\rho,y-z)  \\
&\qquad +1_{\rho <t}  \int_0^1 \frac{d}{d\theta} \left( (D_x^\gamma p) (t,\rho,\theta x+(1-\theta)y-z)  \right) d\theta\bigg|
\end{align*}
Moreover, by \eqref{ker est 1},  \eqref{ker est 2}, and Fubini's theorem,  the above term is less than or equal to 
\begin{align}
										\notag
& N 1_{\rho<s} \left| |t-s| |s-\rho|^{-d/2-|\gamma|/2 -1} \exp\left( -c_0 |s-\rho|^{-1} |y-z|^{2} \right) \right| \\
								\notag
& + N1_{ s\leq \rho <t} \left|  |t-\rho|^{-d/2-|\gamma|/2} \exp\left( -c_0 |t-\rho|^{-1} |y-z|^{2} \right) \right| \\
								\label{20200113}
&+  N 1_{\rho<t} \int_0^1 |x-y| |t-\rho|^{-d/2-(|\gamma|+1)/2} \exp\left( -c_0 |t-\rho|^{-1} \left(|\theta x +(1-\theta) y-z|   \right)^{2} \right)  d\theta.
\end{align}
It is obvious that
\begin{align*}
I_1(t_0,c) \bigcap   \left\{ \rho \in \fR :    t > \rho \geq s  \right\} = \emptyset. 
\end{align*}
Then by \eqref{2019110101}, simple change of variables, the Fubini theorem,  \eqref{20200113}, and \eqref{2019110301}, 
\begin{align*}
&\int_{ I_1}  \left[ \int_{A(\rho,t_0,x_0)} | (D_x^\gamma p)(t,\rho,x-z) -  (D_x^\gamma p)(s,\rho,y-z)| dz \right]^{r} d\rho  \\
&\leq \int_{ I_1}  \left[ \int_{\fR^d} | (D_x^\gamma p)(t,\rho,x-z) -  (D_x^\gamma p)(s,\rho,y-z)| dz \right]^{r} d\rho  \\
&\leq N \left( \int_{4c^2}^\infty  |t-s|^r \rho^{-r (|\gamma|/2 + 1) } + |x-y|^r \rho^{-r(|\gamma|+1)/2}  \right) \\
&\leq N \left( |t-s|^r c^{-r (|\gamma| + 2)  + 2} + |x-y|^r c^{-r(|\gamma|+1)  +2} \right) \\
&\leq N \left(  c^{-r |\gamma|  + 2} \right) \leq N(d,\kappa,K,r, |\gamma|).
\end{align*}
On the other hand, by \eqref{2019110102}, \eqref{20200113}, simple change of variables, and th Fubini theorem, 
\begin{align*}
&\int_{ I_2}  \left[ \int_{A(\rho,t_0,x_0)} |(D_x^\gamma p)(t,\rho,x-z) -   (D_x^\gamma p)(s,\rho,y-z)| dz \right]^{r} d\rho  \\
&\leq N  |t-s|^r \int_{|\rho -t_0| \leq (4c)^2} |s-\rho|^{- r(|\gamma|+2)/2} \left[ \int_{|z| \geq 4|s-\rho|^{-1/2}c}   \exp\left( -c_0 |z|^{2} \right) dz\right]^r d\rho \\
&\quad +N   \int_{|\rho -t_0| \leq (4c)^2} |t-\rho|^{-r |\gamma|/2}  \left[ \int_{|z| \geq 4|t-\rho|^{-1/2}c} \left|  \exp\left( -c_0 |z|^{2} \right) \right| dz \right]^r d\rho \\
&\quad +N|x-y|^r \int_{|\rho -t_0| \leq (4c)^2} |t-\rho|^{- r(|\gamma|+1)/2}   \left[ \int_{|z| \geq 4|t-\rho|^{-1/2}c} \left|  \exp\left( -c_0 |z|^{2} \right) \right| dz \right]^r d\rho \\
&\leq N  |t-s|^r c^{-r(|\gamma|+2) +2}  \left[ \int_{\fR^d}   |z|^{(|\gamma|+2)}\exp\left( -c_0 |z|^{2} \right)  dz\right]^r  \\
&\quad +N   c^{-r|\gamma| +2} \left[ \int_{\fR^d}  |z|^{|\gamma|} \left|  \exp\left( -c_0 |z|^{2} \right) \right| dz \right]^r \\
&\quad +N|x-y|^r  c^{-r(|\gamma|+1) +2}   \left[ \int_{\fR^d} |z|^{|\gamma|+1} \left|  \exp\left( -c_0 |z|^{2} \right) \right| dz \right]^r d\rho \\
&\leq N(d,\kappa,K,r, |\gamma|).
\end{align*}
The lemma is proved. 
\end{proof}
Since 
$$
\fR^{d+1} \setminus Q_{8c}(t_0,x_0)  
 \subset \left\{ (\rho,z) \in \fR^{d+1} :  |\rho -t_0|^{1/2}   + |z-x_0|  \geq 8c  \right\},
$$
one can easily check that the following two corollaries are particular cases of Lemma \ref{hor lem}.
\begin{corollary}[H\"ormander's condition]
						\label{cor hor 1}
There exists a positive constant $N=N(d, \kappa, K)$ such that 
\begin{align*}
\int_{   \fR^{d+1} \setminus Q_{8c}(t_0,x_0)  } | p_{xx}(t,\rho,x-z) -  p_{xx}(s,\rho,y-z)| dz d\rho \leq N
\end{align*}
for all $c \in (0,\infty)$, $(t_0,x_0) \in (0,\infty) \times \fR^d$, and  $(t,x), (s,y) \in Q_c(t_0,x_0)$.
\end{corollary}

\begin{corollary}[Stochastic H\"ormander's condition]
					\label{cor hor 2}
There exists a positive constant $N=N(d, \kappa, K)$ such that 
\begin{align*}
\int_{  -\infty }^\infty  \left[ \int_{(\rho, z) \in \fR^{d+1} \setminus Q_{8c}(t_0,x_0)  } \left|  p_x(t,\rho,x-z) -    p_x (s,\rho,y-z) \right| dz \right]^{2} d\rho 
\leq N
\end{align*}
for all $c \in (0,\infty)$, $(t_0,x_0) \in (0,\infty) \times \fR^d$, and  $(t,x), (s,y) \in Q_c(t_0,x_0)$.
\end{corollary}

\begin{remark}
							\label{rmk 20200129}
Assume that our coefficients $a^{ij}(\omega,t)$ are random, i.e.  
\begin{align*}
\kappa |\xi|^2 \leq a^{ij}(\omega,t) \xi^i \xi^j \leq K |\xi|^2 \qquad  \forall (\omega, t,\xi) \in \Omega \times [0,\infty) \times \fR^d.
\end{align*}
Then applying Corollary \ref{cor hor 1} for each $\omega$, we have
\begin{align}
									\label{20200126}
\sup_{\omega \in \Omega} \int_{   \fR^{d+1} \setminus Q_{8c}(t_0,x_0)  } | p_{xx}(\omega, t,\rho,x-z) -  p_{xx}(\omega, s,\rho,y-z)| dz d\rho \leq N
\end{align}
for all $c \in (0,\infty)$, $(t_0,x_0) \in (0,\infty) \times \fR^d$, and  $(t,x), (s,y) \in Q_c(t_0,x_0)$.
Moreover, our random kernel $p(\omega,t,\rho,x)$ satisfies \eqref{ker est 1} and \eqref{ker est 2} uniformly for $\omega$, i.e. 
\begin{align*}
\left|D_x^\gamma p(\omega,t,\rho,x)  \right| \leq N \left(  (t-\rho)^{-d/2 -\frac{\gamma}{2}} \right) \exp\left( -c_0 (t-\rho)^{-1}|x|^2\right) 
\end{align*}
and
\begin{align*}
\left| D_t D_x^\gamma p(\omega, t,\rho,x)  \right| \leq N \left(  (t-\rho)^{-\frac{d}{2} -\frac{\gamma}{2} -1} \right) \exp\left( -c_0 (t-\rho)^{-1}|x|^2\right)
\end{align*}
for all $\omega \in \Omega,~ t> \rho, ~ x \in \fR^d$.
Thus following the proof of  Lemma \ref{hor lem}, we have 
\begin{align*}
&\left|(D^\gamma_x p)(\omega, t,\rho,x-z) -   (D^\gamma_x p)(\omega, s,\rho,y-z) \right| \\
& \leq N 1_{\rho<s} \left| |t-s| |s-\rho|^{-d/2-|\gamma|/2 -1} \exp\left( -c_0 |s-\rho|^{-1} |y-z|^{2} \right) \right| \\
& + N1_{ s\leq \rho <t} \left|  |t-\rho|^{-d/2-|\gamma|/2} \exp\left( -c_0 |t-\rho|^{-1} |y-z|^{2} \right) \right| \\
&+  N 1_{\rho<t} \int_0^1 |x-y| |t-\rho|^{-d/2-(|\gamma|+1)/2} \exp\left( -c_0 |t-\rho|^{-1} \left(|\theta x +(1-\theta) y-z|   \right)^{2} \right)  d\theta.
\end{align*}
By using these $\omega$-uniform estimates, one can easily prove that
\begin{align}
							\label{2020012602}
 \int_{   \fR^{d+1} \setminus Q_{8c}(t_0,x_0)  } \sup_{\omega \in \Omega} | p_{xx}(\omega,t,\rho,x-z) -  p_{xx}(\omega,s,\rho,y-z)| dz d\rho \leq N,
\end{align}
which is obviously better than \eqref{20200126}.
\end{remark}

\mysection{A maximal regularity moment estimate for second-order PDEs without stochastic noises}

In this section, we handle a random equation without stochastic noises. 
We study the following equation 
\begin{align}
						\label{det ran eqn}
				 du= \left( \bar a^{ij}(\omega, t)u_{x^ix^j}+ f(\omega)  \right)dt , \quad t \in (0,T); \quad u(0,\cdot)=0. 
\end{align}
Here we present the random parameter $\omega$ to emphasize that our coefficients $\bar a^{ij}$ and free data $f$ are random in this section.

Let $T \in (0,\infty]$. We introduce nice  subspaces of $\bH_{p,r}^\gamma(T)$ and $\tilde \bH_{p,r}^\gamma(T,l_2)$ which will be used in approximations.
We denote by $\bH_c^\infty( T )$ the space of all stochastic processes $f$ such that 
$$
f(t,x)= \sum_{i=1}^{j}1_{A_i}(\omega) 1_{B_i} (t)f^{i}(x),
$$
where $j \in \bN$, $f^{i} \in C_c^\infty(\fR^d)$, $A_i \in \rF$, and $B_i \in \cB((0,T))$ with $|B_i| < \infty$. 
Moreover, we denote by $\tilde \bH_c^\infty(T, l_2)$ the space of all stochastic processes $g=(g^1,g^2, \ldots)$ such that $g^k=0$ for all large $k$ and each $g^k$ is  of the type
$$
g^k(t,x)= \sum_{i=1}^{j(k)}1_{(\tau_{i-1},\tau_i]}(t) g^{ik}(x),
$$
where  $j(k) \in \bN$, $g^{ik} \in C_c^\infty(\fR^d)$, and $\tau_i$ are bounded stopping times with $\tau_i \leq T$ on the probability space.

Obviously, for all $ 0 < T_1 \leq T_2 \leq \infty$, $p,r \in (0,\infty]$, and $\gamma \in \bZ_+$, we have
\begin{align*}
\bH_c^\infty(T_1) \subset \bH_c^\infty(T_2) \subset \bH_c^\infty(\infty) \subset \bH_{p,r}^\gamma(\infty)
\end{align*}
and
\begin{align*}
\tilde \bH_c^\infty(T_1,l_2) \subset \tilde \bH_c^\infty(T_2,l_2) \subset  \tilde \bH_c^\infty(\infty,l_2) \subset \tilde \bH_{p,r}^\gamma(\infty,l_2).
\end{align*}
\begin{lemma}
						\label{lem dense}
\begin{enumerate}[(i)]
Let  $T \in (0,\infty]$, $p,r \in (1,\infty)$, and $\gamma \in \bZ_+$.
\item $\bH_c^\infty(T)$ is dense in $ \bH_{p,r}^\gamma(T)$.
\item $\tilde \bH_c^\infty(T,l_2)$  is dense in $\tilde \bH_{p,r}^\gamma(T,l_2)$.
\end{enumerate}

\end{lemma}
\begin{proof}
It is well-known that $\tilde \bH^\infty_c(T,l_2) $ is dense in  $\tilde \bH^\gamma_p(T,l_2)$  for all $p \in (1,\infty)$ and $\gamma \in \bZ_+$ (see \cite[Theorem 3.10]{Krylov1999}).
Recall that  $\tilde \bH^\gamma_{p}(T)$, $\tilde \bH^\gamma_{p}(T,l_2)$, $\tilde \bH^\gamma_{p,r}(T)$, $\tilde \bH^\gamma_{p,r}(T,l_2)$ are reflexive Banach spaces 
 for all $T \in (0,\infty]$, $p,r \in (1,\infty)$ and $\gamma \in \bZ_+$ (Remark \ref{re banach}).
Therefore following the idea of the proof of \cite[Theorem 3.10]{Krylov1999} (an application of Riesz's theorem), 
one can also easily check that $\bH_c^\infty(T)$ and $\tilde \bH_c^\infty(T,l_2)$  are dense in $\bH_{r,p}^\gamma(T)$ and $\tilde \bH_{r,p}^\gamma(T,l_2)$ 
for all $p,r \in (1,\infty)$ and $\gamma \in \bZ_+$,  respectively.  The lemma is proved. 
\end{proof}

\begin{theorem}
					\label{thm 3-1}
Let $p \in (0,\infty]$, $r \in (1,\infty)$, and $T \in (0,\infty)$.
Assume that the coefficients $\bar a^{ij}(t)$ are nonnegative symmetric matrix-valued $\rF \times \cB((0,\infty))$-measurable and satisfy the following ellipticity condition
\begin{align*}
\kappa |\xi|^2 \leq \bar a^{ij}(\omega, t) \xi^i \xi^j \leq K |\xi|^2 \qquad  \forall (\omega, t,\xi) \in \Omega \times [0,\infty) \times \fR^d.
\end{align*}
Then for any $f \in L_p\left(\Omega, \rF ,  P; L_r\left( (0,T), \cB((0,T)) ;L_r \right) \right)$, there exists a unique solution $u$ to equation \eqref{det ran eqn} such that
\begin{align}
					\label{2019100802}
\|u\|_{L_p\left(\Omega, \rF ,  P; L_r\left( (0,T), \cB((0,T)) ;L_r \right) \right)}
\leq N_1 \|f\|_{L_p\left(\Omega, \rF ,  P; L_r\left( (0,T), \cB((0,T)) ;L_r \right) \right)},
\end{align}
and
\begin{align}
					\label{2019100802-2}
\|u_{xx}\|_{L_p\left(\Omega, \rF ,  P; L_r\left( (0,T), \cB((0,T)) ;L_r \right) \right)}
\leq N_2 \|f\|_{L_p\left(\Omega, \rF ,  P; L_r\left( (0,T), \cB((0,T)) ;L_r \right) \right)},
\end{align}
where $N_1=N_1(d,p,r,\kappa, K,T)$ and $N_2=N_2(d,p,r,\kappa, K)$. 
\end{theorem}
\begin{proof}
Note that 
$$
\|f\|^p_{L_p\left(\Omega, \rF ,  P; L_r\left( (0,T), \cB((0,T)) ;L_r \right) \right)}
= \bE \left[  \left(\int_0^T \int_{\fR^d} |f(\omega,t, x)|^r dx dt\right)^{p/r}\right] < \infty
$$
and thus
$f\in L_r( [0,T] \times \fR^d)$ $ (a.s.)$.
Therefore for almost every $\omega \in \Omega$, by the deterministic classical result (cf \cite{Krylov2008,kim2015parabolicps}),  there exists a unique solution $u(\omega,t,x)$ to \eqref{det ran eqn} such that
\begin{align}
						\label{2019100803}
u(\omega,t,x) = \int_0^t \int_{\fR^d} p(\omega,s,x-y)f(\omega,s,y) dyds,
\end{align}
\begin{align}
						\label{2019100801}
\int_0^T \|u(\omega, t,\cdot)\|^r_{H_r^2} dt \leq N_1 \int_0^T \|f(\omega,t,\cdot)\|^r_{L_r} dt ,
\end{align}
and
\begin{align}
						\label{2019100801-2}
\int_0^T \|u_{xx}(\omega, t,\cdot)\|^r_{L_r} dt \leq N_2 \int_0^T \|f(\omega,t,\cdot)\|^r_{L_r} dt ,
\end{align}
where $N_1=N_1(d,p,r,\kappa, K,T)$ and $N_2=N_2(d,p,r,\kappa, K)$. 
Thus taking the $L_{p/r}(\Omega)$-seminorm to the both sides of  \eqref{2019100801} and \eqref{2019100801-2},  we have \eqref{2019100802} and  \eqref{2019100802-2}, respectively.
\end{proof}
\begin{remark}
Let $u$ be a solution to 
 \begin{align}
							\label{2020010701}
				 du= \left( \bar a^{ij}(\omega, t)u_{x^ix^j}+ f(\omega)  \right)dt , \quad t \in (0,\infty); \quad u(0,\cdot)=0.
\end{align}
Assume that  $f \in L_p\left(\Omega, \rF ,  P; L_r\left( (0,\infty), \cB((0,\infty)) ;L_r \right) \right)$ and 
$$
\|u\|_{L_p\left(\Omega, \rF ,  P; L_r\left( (0,T), \cB((0,T)) ;L_r \right) \right)} 
+\|u_{xx}\|_{L_p\left(\Omega, \rF ,  P; L_r\left( (0,T), \cB((0,T)) ;L_r \right) \right)} <\infty
$$
for all $T \in (0,\infty)$ with $r \in (1,\infty)$ and $p \in (0,\infty)$.
Then by \eqref{2019100802-2},
\begin{align}
						\label{2019100810-2}
\bE \left[  \left(\int_0^\infty \|u_{xx}(t,\cdot)\|^r_{L_r} dt  \right)^{p/r}\right] \leq N_2 \bE \left[ \left(\int_0^\infty \|f(t,\cdot)\|^r_{L_r} dt \right)^{p/r} \right]
\end{align}
since the constant $N_2$ in \eqref{2019100802-2} is independent of $T$.
\end{remark}

\begin{theorem}
				\label{thm 3-2}
Let $1<r \leq p < \infty$. 
Assume that the coefficients $a^{ij}(t)$ are $\rF \times \cB((0,\infty))$-measurable and satisfy the following ellipticity condition
\begin{align*}
\kappa |\xi|^2 \leq \bar a^{ij}(\omega, t) \xi^i \xi^j \leq K |\xi|^2 \qquad  \forall (\omega, t,\xi) \in \Omega \times [0,\infty) \times \fR^d.
\end{align*}
Then for any $f \in \bL_{p,r}(T)$, there exists a unique solution $u \in \bH_{p,r}^2(T)$ to equation \eqref{det ran eqn} such that
\begin{align}
								\label{2020011001}
\|u\|_{\bH_{p,r}^2(T)}\leq N_1 \|f\|_{\bL_{p,r}(T)}.
\end{align}
and
\begin{align}
								\label{2020011001-2}
\|u_{xx}\|_{\bH_{p,r}(T)}\leq N_2 \|f\|_{\bL_{p,r}(T)},
\end{align}
where $N_1=N_1(d,p,r,\kappa,K,T)$ and $N_2=N_2(d,p,r,\kappa,K)$.
\end{theorem}

The proof of the theorem will be given in Section \ref{pf thm 3-2}.

\mysection{A boundedness of integral operators }

There are close relations between boundedness of integral operators and estimates for solutions to \eqref{main eqn} since solutions can be represented by integral operators. 
We show the boundedness of many integral operators which play important roles to obtain the maximal regularity moment estimate.

In this section, we handle both the deterministic coefficients $ a^{ij}(t)$ on $(0, \infty) $ and random coefficients $\bar a^{ij}(\omega,t)$ on $\Omega \times (0, \infty) $ satisfying
\begin{align*}
\kappa |\xi|^2 \leq  a^{ij}(t) \leq K |\xi|^2 \qquad  \forall (t,\xi) \in [0,\infty) \times \fR^d
\end{align*}
and
\begin{align}
				\label{ran ellip}
\kappa |\xi|^2 \leq \bar a^{ij}(\omega,t) \leq K |\xi|^2 \qquad  \forall (\omega, t,\xi) \in \Omega \times [0,\infty) \times \fR^d.
\end{align}
We assume that $ a^{ij}(t)$ is nonnegative symmetric matrix-valued $\cB((0,\infty))$-measurable and $\bar a^{ij}(\omega,t)$ is $\rF \times \cB((0,\infty))$-measurable.
Recall
$$
p(t,\rho,x) = 1_{\rho<t} (4\pi)^{-d/2} (\det B_{t \rho} )^{1/2} \exp ( - (B_{t \rho} x, x)/4 ),
$$
where
$$
a(t)=(a^{ij}(t)), \quad A_{ t \rho}= \int_\rho^t a(\eta) d\eta, \quad B_{t \rho}= A^{-1}_{t \rho}, \quad \sigma_{t \rho}=A^{1/2}_{t \rho}.
$$
For random coefficients $a^{ij}(\omega,t)$, we similarly define
\begin{align}
					\label{ran ker def}
\bar p(\omega,t,\rho,x) = 1_{\rho<t} (4\pi)^{-d/2} (\det B_{\omega,t \rho} )^{1/2} \exp ( - (B_{\omega,t \rho} x, x)/4 ),
\end{align}
where
$$
\bar a(\omega,t)=(\bar a^{ij}(\omega, t)), \quad A_{ \omega, t \rho}= \int_\rho^t \bar a(\omega, \eta) d\eta, \quad B_{\omega,t \rho}= A^{-1}_{\omega, t \rho}, \quad \sigma_{\omega, t \rho}=A^{1/2}_{\omega, t \rho}.
$$
Note that fixing $\omega \in \Omega$, we have kernel estimates for  $\bar p(\omega, t,\rho,x)$ from \eqref{ker est 1} and \eqref{ker est 2}.
Moreover these kernel estimates hold uniformly for all $\omega \in \Omega$ since \eqref{ran ellip} holds uniformly for $\omega$. 

Next we introduce an integral operator with a random kernel and a stochastic integral operator with a non-random kernel. 

For $f \in \bL_{p,r}(T)$ and  $g \in \bL_{p,r}(T,l_2)$, we denote
$$
\cT f(t,x) = \int_0^t \int_{\fR^d} \bar p(\omega, t,\rho,x-y)f(\rho,y) dy d\rho
$$
and
\begin{align}
							\label{2020012010}
\bT f(t,x) = \int_0^t \int_{\fR^d} p(t,\rho,x-y)g^k(\rho,y) dy dw^k_\rho.
\end{align}
\begin{remark}
							\label{rmk 20200123}
Note that for each $\omega,t,x$, the function  $\bar p(\omega,t,\rho,x)$ is not $\rF_\rho$-adapted (actually it is $\rF_{t}$-adapted)  even if we give the predictable assumption on $a^{ij}(\omega,t)$.
Thus if we consider the random kernel  $\bar p(\omega,t,\rho,x-y)$ in \eqref{2020012010}, then the operator $\bT$ is not well-defined since  
the predictability of the stochastic process $p(\omega, t,\rho,x-y)g^k(\rho,y)$ is not guaranteed even though $g$ is a very nice stochastic process. 
In other words,  for each $t,\rho,x$, there is no guarantee that $\int_{\fR^d} \bar p(\omega,t,\rho,x-y)g^k(\rho,y) dy $ is $\rF_\rho$-adapted. 
Therefore $\int_{\fR^d} \bar p(\omega,t,\rho,x-y)g^k(\rho,y) dy$ is not stochastic integrable. 
\end{remark}

We show that $\cT$ is a bounded operator on $\bL_{p,r}(T)$ for all $r,p \in [1,\infty]$ first. 
\begin{lemma}
							\label{lem20200110}
Let $T \in (0,\infty)$ and $p,r \in [1,\infty]$. 
Then  there exists a positive constant $N(d,p,\kappa, K,T)$ such that
$$
\|\cT f\|_{\bL_{p,r}(T)}\leq N \|f\|_{\bL_{p,r}(T)} \qquad \forall f \in \bL_{p,r}(T).
$$
\end{lemma}
\begin{proof}
Since the proof of the case $p=\infty$ or $r=\infty$ is easier, we only prove the lemma with the additional assumption $ p,r < \infty$.
By the generalized Minkowski inequality and \eqref{ker est 1}, we have
\begin{align*}
&\|\cT f\|^p_{\bL_{p,r}(T)} \\
&=\int_0^T \int_{\fR^d} \left(\bE \left[  \left|\int_0^t \int_{\fR^d} \bar  p(\omega,t,\rho , y)  f(\rho,x-y) dy d\rho \right|^r \right] \right)^{p/r}dx dt  \\
&\leq \int_0^T \int_{\fR^d}  \left[  \left|\int_0^t \int_{\fR^d}  \bE \left[| \bar p(\omega,t,\rho , y)  f(\rho, x-y)|^r \right]^{1/r} dy d\rho \right| \right]^{p}dx dt \\
&\leq \int_0^T \int_{\fR^d}  \bigg[  \bigg|\int_0^t \int_{\fR^d}  (t-\rho)^{-d/2 }\exp\left( -c_0 (t-\rho)^{-1}|y|^2\right)  \\
&\qquad \qquad \qquad \qquad \qquad \qquad \qquad \qquad \qquad \qquad  \times \bE \left[|f(\rho, x-y)|^r \right]^{1/r} dy d\rho \bigg| \bigg]^{p}dx dt  \\
&=\int_0^T \int_{\fR^d}   \left[  \left|\int_0^t \int_{\fR^d}  \rho^{-d/2 }\exp\left( -c_0 \rho^{-1}|y|^2\right)   \bE \left[|f(t-\rho, x-y)|^r \right]^{1/r} dy d\rho \right| \right]^{p}dx dt  \\
&\leq  \left[ \int_0^t \int_{\fR^d}  \rho^{-d/2 }\exp\left( -c_0 \rho^{-1}|y|^2\right)  dy d\rho  \right]^{p} \|f\|_{\bL_{p,r}(T)}^p \\
&\leq N \|f\|_{\bL_{p,r}(T)}^p.
\end{align*}
The lemma is proved. 
\end{proof}
Similarly $\bT$ is a bounded operator from  $\tilde \bL_{p,r}(T,l)$ to $ \bL_{p,r}(T)$ with some restrictions on the range of $p$ and $r$. 
\begin{lemma}
						\label{lem2020011010}
Let $T \in (0,\infty)$, $p \in [2,\infty]$, and $r \in [2,\infty)$. 
Then  there exists a positive constant $N(d,p,\kappa, K,T)$ such that
\begin{align}
						\label{20200114}
\|\bT g\|_{\bL_{p,r}(T)}
\leq N \|g\|_{\bL_{p,r}(T)} \quad \forall g \in \tilde \bL_{p,r}(T,l_2).
\end{align}
\end{lemma}
\begin{proof}
As in the proof of the previous lemma, we only prove the more difficult case $ p < \infty$.
Let $g \in  \tilde \bL_{p,r}(T,l_2)$.
By the Burkholder-Davis-Gundy inequality, 
\begin{align*}
&\|\bT g\|^p_{\bL_{p,r}(T)} \\
&=\int_0^T \int_{\fR^d} \left(\bE \left[  \left|\int_0^t \int_{\fR^d} p(t,\rho , y)  g^k(\rho,x-y) dy dw^k_\rho \right|^r \right] \right)^{p/r}dx dt  \\
&\leq N \int_0^T \int_{\fR^d}  \left(\bE \left[  \left( \int_0^t \left| \int_{\fR^d} p(t,\rho , y)  g(\rho,x-y) dy\right|_{l_2}^2 d\rho \right)^{r/2}  \right] \right)^{p/r}dx dt \\
\end{align*}
Moreover, by the generalized Minkowski inequality,  the H\"older inequality, and the Fubini theorem, the above term is less than or equal to 
\begin{align*}
&N \int_0^T \int_{\fR^d}  \left(\int_0^t  \left( \bE \left[   \left| \int_{\fR^d} p(t,\rho , y)  g(\rho,x-y) dy\right|_{l_2}^r  \right]  \right)^{2/r} d\rho \right)^{p/2}dx dt \\
&\leq N \int_0^T \int_{\fR^d}  \left(\int_0^t  \left(  \int_{\fR^d}  \left( \bE \left[   \left| p(t,\rho , y)  g(\rho,x-y)\right|_{l_2}^r \right]  \right)^{1/r}  dy  \right)^{2} d\rho \right)^{p/2}dx dt \\
&\leq N \int_0^T \int_{\fR^d}  \int_0^t  \left(  \int_{\fR^d}   \left( \bE \left[   \left|p(t,\rho , y)  g(\rho,x-y)\right|_{l_2}^r \right]  \right)^{1/r}  dy  \right)^p d\rho dx dt.
\end{align*}
Applying \eqref{ker est 1}  and the Minkowski inequality, we show that the above term is less than or equal to 
\begin{align*}
&N\int_0^T  \int_0^T 1_{0<\rho <t} \Bigg( \int_{\fR^d}    (t-\rho)^{-d/2 }\exp\left( -c_0 (t-\rho)^{-1}|y|^2\right)  \\
& \qquad \qquad \qquad \qquad \qquad \qquad \qquad  \left( \int_{\fR^d}    \left( \bE \left[   \left|g(\rho,x-y)\right|_{l_2}^r \right]  \right)^{p/r}   dx  \right)^{1/p}  dy  \Bigg)^p d\rho dt \\
&\leq N\int_0^T  \int_0^T  \Bigg( \int_{\fR^d}     \rho^{-d/2 }\exp\left( -c_0 \rho^{-1}|y|^2\right)   \\
& \qquad \qquad \qquad \qquad \qquad \left( \int_{\fR^d}   \left( 1_{0<t-\rho}\bE \left[   \left|g(t-\rho,x-y)\right|_{l_2}^r \right]  \right)^{p/r}   dx  \right)^{1/p}  dy  \Bigg)^p dt d\rho.
\end{align*}
Finally, Minkowski's inequality and the following $\rho$-uniform inequality
$$
\int_{\fR^d} \rho^{-d/2 }\exp\left( -c_0 \rho^{-1}|y|^2\right) dy \leq N(d,\kappa,K)
$$ 
imply that the last term in the above inequalities is controlled by 
\begin{align*}
& N\int_0^T \Bigg(  \int_{\fR^d}     \rho^{-d/2 }\exp\left( -c_0 \rho^{-1}|y|^2\right)     \\
& \qquad \qquad \qquad \qquad  \Bigg(\int_0^T\int_{\fR^d}   \left( 1_{0<t-\rho}\bE \left[   \left|g(t-\rho,x-y)\right|_{l_2}^r \right]  \right)^{p/r}   dx  dt \Bigg)^{1/p}  dy  \Bigg)^{p} d\rho  \\
&\leq N\int_0^T \left(  \int_{\fR^d} \rho^{-d/2 }\exp\left( -c_0 \rho^{-1}|y|^2\right)  dy  \right)^{p} d\rho   \|g\|^p_{\bL_{p,r}(T,l_2)}\\
&\leq N  \|g\|^p_{\bL_{p,r}(T,l_2)}.
\end{align*}
The lemma is proved. 
\end{proof}
\begin{remark}
The Burkholder-Davis-Gundy inequality gives the equivalence of two norms
\begin{align*}
\left\|\int_0^t \int_{\fR^d} p(t,\rho , y)  g^k(\rho,x-y) dy dw^k \right\|_{L_r(\Omega)}
\end{align*}
and
\begin{align*}
\left\| \left(\int_0^t  \left|\int_{\fR^d} p(t,\rho , y)  g^k(\rho,x-y) dy\right|_{l_2}^2 d\rho \right)^{1/2} \right\|_{L_r(\Omega)}
\end{align*}
for all $ r \in (0,\infty)$. 
Since the proof of Lemma \ref{lem2020011010} heavily depends on this equivalence, \eqref{20200114} is not expected if $r=\infty$.
Moreover if $r \in (0,2)$, then we cannot apply the generalized Minkowski inequality for the exponent $r/2$. 
\end{remark}

Before introducing some singular integral operators related to our maximal moment estimates, we present analytic tools to control singularities first. 
BMO (bounded mean oscillation) estimates have been known very important tools to treat singularities of integral operators (cf. \cite{grafakos2008classical,grafakos2009modern,Stein1993}). 
Thus we briefly review the definition of a BMO space.
For a locally integrable function $h$ on $\fR^{d+1}$, we define the BMO semi-norm of $h$ on $\fR^{d+1}$ as follows:
$$
\|h\|_{BMO} := \sup_{Q}  \aint_Q |h(\rho,z) - h_Q| d\rho dz:= \sup_{Q} \frac{1}{|Q|} \int_Q |h(\rho,z) - h_Q| d\rho dz,
$$
where
$$
h_Q:= \aint_{Q} f(\rho,z) d\rho dz := \frac{1}{|Q|} \int_Q f(r,z)d\rho dz
$$
and the sup is taken over all cylinders $Q$ of the type
\begin{align*}
Q
&:=Q_c(t_0,x_0) \\
&:= (t_0 -c^2, t_0+c^2) \times B_c(x_0) \\
&:= (t_0 -c^2, t_0+c^2) \times \{ x \in \fR^d : |x-x_0| < c\}, \quad c>0, ~(t_0,x_0) \in \fR^{d+1}.
\end{align*}
Moreover, we define the Feffereman-Stein sharp function and Hardy-Littlewood maximal function related to these cylinders. 
For $(t,x) \in \fR^{d+1}$, define
$$
h^\sharp  (t,x)= \sup_{Q} \frac{1}{|Q|} \int_Q |h(\rho,z) - h_Q| d\rho dz
$$
and
$$
\cM h (t,x) =\sup_{Q} \frac{1}{|Q|} \int_Q |h(\rho,z)| d\rho dz,
$$
the sup is taken over all cylinders $Q$ containing $(t,x)$. 
It is well-known (\cite{Stein1993}) that for all $p \in (1,\infty)$,  the $L_p$-norms of $h$, $h^\sharp$, and $\cM h$ are equivalent. In other words, for any $p \in (1,\infty)$, there exist positive constants $N_1(d,p)$, $N_2(d,p)$, $N_3(d,p)$ such that
\begin{align}
					\label{fs thm}
\|h\|_{L_p( \fR^{d+1})}  \leq N_1\|h^\sharp\|_{L_p( \fR^{d+1})}  \leq   N_2\|\cM h\|_{L_p( \fR^{d+1})} \leq N_3\|h\|_{L_p( \fR^{d+1})}.
\end{align}

Finally we introduce our singular integral operators which are not linear (but sublinear). 
For $f \in  \bH_c^\infty(\infty)$, we denote
\begin{align}
						\label{2019092501}
\cG_r [f](t,x) :=\left(\bE \left[  \left|\int_{0}^\infty \int_{\fR^d} \bar p_{xx}(\omega,t,\rho , z)  f(\rho,x-z) dz d\rho \right|^r \right] \right)^{1/r}
\end{align}
and
\begin{align}
						\notag
G_r [f](t,x,s,y) := \bigg(\bE \bigg[  \bigg|\int_{0}^\infty \int_{\fR^d} \bigg( &\bar p_{xx}(\omega,t,\rho ,z)  f(\rho,x-z)  \\
&- \bar p_{xx}(\omega, s,\rho ,z)  f(\rho,y-z) \bigg) dz d\rho \bigg|^r \bigg] \bigg)^{1/r}.
						\label{2019092701}
\end{align}
Similarly,  for $g \in \tilde \bH_c^\infty(\infty, l_2)$, we denote 
\begin{align}
							\label{2020010640}
\rG_r [g](t,x) :=\left(\bE \left[  \left( \int_{0}^\infty  \left| \int_{\fR^d} p_{x}(t,\rho , z)  g(\rho,x-z) dz \right|_{l_2}^2 d\rho \right)^{r/2} \right] \right)^{1/r}
\end{align}
and
\begin{align}
						\notag						
\bG_r [g](t,x,s,y) := \Bigg(\bE \Bigg[  \bigg|\int_{0}^\infty \bigg| \int_{\fR^d} \bigg( &p_{x}(t,\rho ,z)  g(\rho,x-z)  \\
							\label{2020010641}
&- p_{x}(s,\rho ,z)  g(\rho,y-z) \bigg) dz   \bigg|_{l_2}^2 d\rho \bigg|^{r/2} \Bigg] \Bigg)^{1/r}.
\end{align}
We emphasize  repeatedly that $\cG_r$ and $G_r$ are defined with random kernels but $\rG_r$ and $\bG_r$ are defined with deterministic kernels. 
\begin{remark}
						\label{rmk 3.6}
\begin{enumerate}[(i)]
\item Let $r \in [1,\infty)$. Then for each $(\omega,t,x) \in \Omega \times (0,\infty) \times \fR^{d}$, \eqref{2019092501} is well-defined as an iterated integral since
$ \bar p(\omega, t,\cdot, \cdot)$ is integrable on $(0,t) \times \fR^d$  for each $\omega$ and $t$ even though the kernel $ \bar p_{x^ix^j}(\omega, t,\cdot,\cdot)$ is not integrable on $(0,t) \times \fR^d$ for all $t$ and $\omega$.  
Indeed, by the integration by parts, the generalized Minkowski inequality, and \eqref{ker est 1},
\begin{align*}
&\cG_r [ f](t,x) \\
&=\left(\bE \left[  \left|\int_{0}^\infty \int_{\fR^d} \bar p(\omega,t,\rho , z)  f_{xx}(\rho,x-z) dz d\rho \right|^r \right]\right)^{1/r} \\
&\leq  \int_0^t \int_{\fR^d}  (t-\rho)^{-d/2 }\exp\left( -c_0 (t-\rho)^{-1}|z|^2 \right) \left(\bE \left[  \left|f_{x^ix^j}(\rho,x-z) \right|^r \right]\right)^{1/r} dz d\rho  \\
&\leq  \int_0^t \int_{\fR^d}(t-\rho)^{-d/2 }\exp\left( -c_0 (t-\rho)^{-1}|z|^2 \right) dz d\rho  \sup_{ (\rho,z) \in(0,t)\times \fR^d}   \left(\bE \left[  \left|f_{xx}(\rho,z) \right|^r \right]\right)^{1/r} \\
&<\infty.
\end{align*}
Similarly,  for $r \in [2,\infty)$,  we have
\begin{align*}
&\rG_r [g](t,x)  \\
&=\left(\bE \left[  \left( \int_{0}^\infty  \left|\int_{\fR^d} p(t,\rho , z)  g_x(\rho,x-z) dz \right|_{l_2}^2 d\rho \right)^{r/2} \right] \right)^{1/r} \\
&\leq \left(  \int_{0}^\infty  \left( \bE \left[ \left|\int_{\fR^d} p(t,\rho , z)  g_x(\rho,x-z) dz \right|_{l_2}^r   \right]  \right)^{2/r} d\rho \right)^{1/2} \\
&\leq \left(  \int_{0}^t  \left( \int_{\fR^d}  \left( (t-\rho)^{-d/2 }\exp\left( -c_0 (t-\rho)^{-1}|z|^2 \right)  \bE \left[  \left| g_x(\rho,x-z)   \right|_{l_2}^r   \right]   \right)^{1/r} dz  \right)^{2} d\rho \right)^{1/2} \\
&\leq \left(  \int_{0}^t  \left( \int_{\fR^d}   (t-\rho)^{-d/2 }\exp\left( -c_0 (t-\rho)^{-1}|z|^2 \right)   dz  \right)^{2} d\rho \right)^{1/2}  \\
&\quad \qquad \qquad \qquad  \qquad \qquad \qquad \qquad \qquad \qquad  \times \sup_{(\rho,z) \in (0,t) \times \fR^d}\left( \bE \left[  \left| g_x(\rho,x-z)   \right|_{l_2}^r  \right]  \right)^{1/r} \\
&<\infty.
\end{align*}
Thus  \eqref{2020010640} is well-defined.  As shown in the above inequalities \eqref{2020010640} is well-defined even if we assume the kernel is random.
However, if the kernel is random, then  the following It\^o's isometry 
\begin{align*}
|\rG_2 [g](t,x)|^2  \approx
\bE \left[  \left| \int_{0}^\infty  \left[\int_{\fR^d} \bar p_{x}(\omega, t,\rho , z)  g(\rho,x-z) dz   \right] dw_\rho^k \right|^{2} \right]
\end{align*}
does not hold since $ \int_{\fR^d}\bar p_{x}(\omega,t,\rho , z)  g(\rho,x-z) dz $ is not it\^o integrable (cf. Remark \ref{rmk 20200123}). 

\item For all $f_1, f_2 \in  \bH_c^\infty(\infty)$ and $(t,x), (s,y) \in (0,\infty) \times \fR^{d}$, applying Minkowski's inequality, we have
\begin{align}
						\notag
G_r[f_1+f_2](t,x,s,y) 
&\leq G_r[f_1](t,x,s,y)+ G_r[f_2](t,x,s,y) \\
						\label{2019092801}
&\leq \cG_r [f_1](t,x) + \cG_r[f_1](s,y) + G_r[f_2](t,x,s,y) \\
							\notag
&\leq \cG_r [f_1](t,x) + \cG_r[f_1](s,y) +  \cG_r [f_2](t,x) + \cG_r[f_2](s,y).
\end{align}
Thus \eqref{2019092701} is  well-defined.
Similarly, for all $g_1, g_2 \in \tilde \bH_c^\infty(\infty, l_2)$,
\begin{align}
						\notag
\bG_r[g_1+g_2](t,x,s,y) 
&\leq \bG_r[g_1](t,x,s,y)+ \bG_r[g_2](t,x,s,y) \\
							\label{20200108}
&\leq \rG_r [g_1](t,x) + \rG_r[g_1](s,y) + \bG_r[g_2](t,x,s,y) \\
							\notag
&\leq \rG_r [g_1](t,x) + \rG_r[g_1](s,y) +  \rG_r [f_2](t,x) + \rG_r[g_2](s,y).
\end{align}
Hence \eqref{2020010641} is well-defied.
\end{enumerate}
\end{remark}

In the following theorems, we show that there exist  continuous extensions of $\cG_r$ and $\rG_r$ on $\bL_r(\infty)$ and $\tilde \bL_r(\infty,l_2)$, respectively.  
We use the same notation $\cG_r$ and $\rG_r$ to denote these extensions in Theorem \ref{thm 5-2} and Theorem \ref{thm 5-4} by slightly abusing the notation, respectively.
\begin{theorem}
					\label{thm 5-2}
Let $r \in (1,\infty)$. Then there exists a continuous extension of $\cG_r $ on $\bL_{r}(\infty)$ such that
\begin{align}
							\label{2020010650}
\| \cG_r[f]\|_{L_r((0,\infty) \times \fR^{d})} \leq N \|f\|_{\bL_{r}(\infty)} \qquad \forall f  \in \bL_{r}(\infty),
\end{align}
where  $N=N(d,r,\kappa,K)$.
\end{theorem}
\begin{proof}
Recall that $\bH_c^\infty(\infty)$ is dense in $ \bL_{r}(\infty)$ (Lemma \ref{lem dense}) and observe that  
\begin{align*}
\| \cG_r[f_1] - \cG_r[f_2]\|_{L_r((0,\infty) \times \fR^{d})} \leq \| \cG_r[f_1-f_2]\|_{L_r((0,\infty) \times \fR^{d})}  \qquad \forall f_1,f_2 \in \bH_c^\infty(\infty).
\end{align*}
Thus  it is sufficient to show \eqref{2020010650} for all $f \in \bH_c^\infty(\infty)$ due to the canonical extension of the sublinear operator.
\eqref{2020010650} can be obtained from parabolic BMO estimates (cf. \cite{kim2015parabolicps}). 
Here is an another method to prove this estimate on the basis of solvability of equation \eqref{2020010701}.
Let $f \in \bH_c^\infty(\infty)$. Since $f$ is nice enough, it is easy to check that the solution $u(t,x)$ to \eqref{2020010701} is given by
$$
u(t,x) = \int_{0}^\infty \int_{\fR^d} \bar p(\omega,t,\rho , z)  f(\rho,x-z) dz d\rho.
$$
Obviously,
$$
\bE |u_{xx}(t,x)|^r = |\cG_r(t,x)|^r.
$$
By \eqref{2019100810-2},
\begin{align*}
 \int_0^\infty \int_{\fR^d} |\cG_r f(t,x)|^r dt dx
 &=\bE \int_0^\infty \int_{\fR^d} |u_{xx}(t,x)|^r dt dx \\
&\leq N(d,r,\kappa,K) \bE \int_0^\infty \int_{\fR^d} |f|^r dt dx.
\end{align*}
Therefore we have \eqref{2020010650}. The theorem is proved. 
\end{proof}

\begin{theorem}
					\label{thm 5-4}
Let $r \in [2,\infty)$.
Then there exists a continuous extension of $\rG_r $ on $\tilde \bL_{r}(\infty,l_2)$ such that
\begin{align}
							\label{2020010702}
\| \rG_r[g]\|_{L_p((0,\infty) \times \fR^{d}) } \leq N \|g\|_{\bL_{r}(\infty,l_2)} \qquad \forall g \in \tilde \bL_r(\infty,l_2).
\end{align}
\end{theorem}
\begin{proof}
As mentioned in the proof of Theorem \ref{thm 5-2}, it is sufficient to show \eqref{2020010702} for all $g \in \tilde \bH_c^\infty(\infty,l_2)$. 
By the Burkholder-Davis-Gundy inequality, it is easy to check that 
\begin{align}
								\label{20200115}
|\rG_r [g](t,x)|^r  \approx
\bE \left[  \left| \int_{0}^\infty  \left[\int_{\fR^d} p_{x}(t,\rho , z)  g(\rho,x-z) dz   \right] dw_\rho^k \right|^{r} \right],
\end{align}
that is,  there exists a positive constant $N=N(d,r)$ such that
\begin{align*}
N^{-1} |\rG_r[g] (t,x)|^r \leq \bE\left[\left|\int_0^t \left[ \int_{\fR^d}p_x(t,s,x-y)g^k(s,y)dy \right]dw^k_s \right|^r\right]
\leq N |\rG_r[g](t,x)|^r.
\end{align*}
Thus \eqref{2020010702} can be obtained from the $L_r$-boundedness of stochastic singular integral operator (see \cite{kim2016lpso}).
This is also obtained from Krylov's $L_p$-theory (\cite{Krylov1999}) to the following stochastic PDEs 
 \begin{align}
						\label{2020010703}
				 du= \left( a^{ij}(t)u_{x^ix^j}+ g^k w_t^k  \right)dt , \quad t \in (0,\infty); \quad u(0,\cdot)=0. 
\end{align}
Since $g \in \tilde \bH_c^\infty(\infty,l_2)$, the solution  $u$ to \eqref{2020010703} is given by
\begin{align}
						\label{2020010710}
u(t,x) = \int_{0}^\infty  \int_{\fR^d} p(t,\rho , z)  g(\rho,x-z) dz  dw_\rho^k 
\end{align}
and satisfy
\begin{align}
						\label{2020010711}
\bE  \int_0^\infty \int_{\fR^d} |u_x(t,x)|^r dt dx \leq N \bE  \int_0^\infty \int_{\fR^d} |g(t,x)|_{l_2}^r dt dx
\end{align}
(cf. \cite{Krylov1999,kim2019sharp}).
Finally \eqref{20200115}, \eqref{2020010710}, and \eqref{2020010711} clearly imply \eqref{2020010702}.
\end{proof}

From the definitions of $\cG_r[f]$ and $\rG_r[g]$, for all $f \in \bH_c^\infty(\infty)$ and $\tilde g \in \bH_c^\infty(\infty,l_2)$, we have
$$
\cG_r[f](t,x) = \rG_r[g](t,x) = 0  \qquad \forall (t,x) \in (-\infty,0] \times \fR^d.
$$
Thus we may assume that the extensions $\cG_r[f](t,x)$ and $\rG_r[g](t,x)$ are defined on $\fR^{d+1}$ for all $f \in \bL_r(\infty)$ and $g \in \tilde \bL_r(\infty,l_2)$  by setting
$$
\cG_r[f](t,x) = \rG_r[g](t,x) = 0  \qquad \forall (t,x) \in  (-\infty,0] \times \fR^d.
$$
Moreover, we may assume that $f$ and $g$ are defined on $\Omega \times \fR^{d+1}$ for all $f \in \bL_r(\infty)$ and $g \in \tilde  \bL_r(\infty,l_2)$ by considering the trivial extensions $1_{(0,\infty)}(t)f(t,x)$ and $1_{(0,\infty)}(t)g(t,x)$. 
Furthermore, if $f \in \bH_r^2(\infty)$, then $\cG_r[f](t,x)$ can be understood pointwisely $(a.e.)$ as an iterated integral (cf. Remark \ref{rmk 3.6} (i)), i.e.
\begin{align}
\cG_r [f](t,x) 
									\notag
&=\left(\bE \left[  \left|\int_{0}^\infty \int_{\fR^d} \bar p(\omega,t,\rho , z)  f_{xx}(\rho,x-z) dz d\rho \right|^r \right] \right)^{1/r} \\
									\label{20200127}
&=\left(\bE \left[  \left|\int_{0}^\infty  \left[\int_{\fR^d} \bar p_{xx}(\omega,t,\rho , z)  f(\rho,x-z) dz \right] d\rho \right|^r \right] \right)^{1/r}.
\end{align}
Similarly, if $g \in \bH_r^1(\infty,l_2)$, then 
\begin{align}
\bG_r [g](t,x) 
									\notag
&=\left(\bE \left[  \left|\int_{0}^\infty  \left|\int_{\fR^d} p(t,\rho , z)  g_{x}(\rho,x-z) dz \right|^2_{l_2} d\rho \right|^{r/2} \right] \right)^{1/r} \\
									\label{202001272}
&=\left(\bE \left[  \left|\int_{0}^\infty  \left| \left[\int_{\fR^d} p_{x}(t,\rho , z)  g(\rho,x-z) dz \right] \right|_{l_2} d\rho \right|^{r/2} \right] \right)^{1/r}.
\end{align}
\begin{theorem}
					\label{thm 5}
Let $1<r \leq p <\infty$.
Then there exists a positive constant $N(d,p,r,\kappa,K)$ such that
$$
\| \cG_r[f]\|_{L_p((0,\infty) \times \fR^{d}) } \leq N \|f\|_{\bL_{p,r}(\infty)} \qquad \forall f \in \bL_{r}(\infty) \cap \bL_{ \infty,r}(\infty).
$$
\end{theorem}
\begin{theorem}
					\label{thm 5-3}
Let $2 \leq r \leq p < \infty$.
Then there exists a positive constant $N(d,p,r,\kappa,K)$ such that
$$
\| \rG_r[g]\|_{L_p((0,\infty) \times \fR^{d}) } \leq N \|g\|_{\bL_{p,r}(\infty,l_2)} \qquad \forall g \in \tilde \bL_{r}(\infty,l_2) \cap \tilde \bL_{\infty,r}(\infty,l_2).
$$
\end{theorem}
The proofs of Theorem \ref{thm 5} and Theorem \ref{thm 5-3} will be given in the last of this section.
To prove these theorems, we need the following preliminaries.  
Recall the cylinders $Q_c(t_0,x_0) := \left(t_0 - c^2, t_0+ c^2\right) \times B_c(x_0)$ in $\fR^{d+1}$ with $(t_0,x_0) \in \fR^{d+1}$ and $c>0$.  
\begin{lemma}
						\label{lem20200108}
Let $r \in [1,\infty)$, $c \in (0,\infty)$, $(t_0,x_0) \in \fR^{d+1}$, and $(t_1, x_1) \in Q_c(t_0,x_0)$. Then for all $f_1, f_2 \in \bL_r(\infty)$, 
\begin{align}
							\notag
& \aint_{Q_c(t_0,x_0)} \aint_{Q_c(t_0,x_0)} | \cG_r \left[ f_1 +f_2 \right](t,x) - \cG_r \left[f_1+f_2 \right](s,y) | dtdxdsdy \\
							\label{2019092802}
&\qquad \leq 2 \cM ( \cG_r [f_1] ) (t_1,x_1)  + \aint_{Q_c(t_0,x_0)} \aint_{Q_c(t_0,x_0)} G_r[f_2](t,x,s,y) dtdxdsdy. 
\end{align}
\end{lemma}
\begin{proof}
Let $f_1, f_2 \in \bL_r(\infty)$.

\vspace{2mm}
{\bf Step 1} We additionally assume that $f_1, f_2 \in \bH_r^2(\infty)$. 
\vspace{2mm}

Let $(t,x), (s,y) \in Q_c(t_0,x_0)$. Due to \eqref{20200127}, $(i=1,2)$
\begin{align*}
\cG_r [f_i](t,x) 
&=\left(\bE \left[  \left|\int_{0}^\infty  \left[\int_{\fR^d} \bar p_{xx}(\omega,t,\rho , z)  f_i(\rho,x-z) dz \right] d\rho \right|^r \right] \right)^{1/r}.
\end{align*}
Then by Minkowski's inequality and \eqref{2019092801},
\begin{align*}
&| \cG_r \left[ f_1 +f_2 \right](t,x) - \cG_r \left[f_1+f_2 \right](s,y) | \\
&\leq |G_r[f_1+f_2](t,x,s,y)| \\
&\leq \cG_r[f_1](t,x) + \cG_r[f_1](s,y) + G_r[f_2](t,x,s,y).
\end{align*}
Taking the mean average  to both sides above, we have
\begin{align*}
& \aint_{Q_c(t_0,x_0)} \aint_{Q_c(t_0,x_0)} | \cG_r \left[ f_1 +f_2 \right](t,x) - \cG_r \left[f_1+f_2 \right](s,y) | dtdxdsdy \\
&\qquad \leq 2 \aint_{Q_c(t_0,x_0)} f_1(t,x)dtdx  + \aint_{Q_c(t_0,x_0)} \aint_{Q_c(t_0,x_0)} G_r[f_2](t,x,s,y) dtdxdsdy .
\end{align*}
Since it is obvious that
$$
 \aint_{Q_c(t_0,x_0)} f_1(t,x)dtdx \leq \cM( \cG_r[f_1]) (t_1,x_1),
 $$
we have \eqref{2019092802}.

\vspace{2mm}
{\bf Step 2} (General case) $f_1, f_2 \in \bL_r(\infty)$.
\vspace{2mm}

We use the Sobolev mollifiers and approximations. Choose a nonnegative $\phi \in C_c^\infty(\fR^d)$ so that
$$
\int_{\fR^d}\phi(x)dx = 1.
$$
For $\varepsilon>0$, define
$$
\phi^\varepsilon (x) = \frac{1}{\varepsilon^d} \phi( x/\varepsilon)
$$
and $(i=1,2)$ 
$$
f_i^{\varepsilon}(t,x) :=  f_i(t,\cdot) \ast \phi^\varepsilon (\cdot) (x) := \int_{\fR^d} f(t,y) \phi(x-y)dy. 
$$
On the basis of properties of Sobolev mollifiers, it is easy to check that 
$$
\|f_i^\varepsilon - f_i\|_{\bL_r(\infty)} \to 0 \quad \text{as} \quad \varepsilon \downarrow 0.
$$
Moreover, \eqref{fs thm} guarantees the continuity of the operator $ f \to \cM f$ on $\bL_r(\infty)$.
Remind that  \eqref{2019092802} holds with $f_1^\varepsilon $ and $f_2^\varepsilon$ for all $\varepsilon>0$  by Step 1. 
Finally taking $\varepsilon \downarrow 0$ we have \eqref{2019092802} for $f_1, f_2 \in \bL_r(\infty)$. 
The lemma is proved. 
\end{proof}

\begin{lemma}
Let $r \in [2,\infty)$, $c \in (0,\infty)$, $(t_0,x_0) \in \fR^{d+1}$, and $(t_1, x_1) \in Q_c(t_0,x_0)$. 
Then for all  $g_1, g_2 \in \tilde \bL_r(\infty,l_2)$, we have
\begin{align*}
& \aint_{Q_c(t_0,x_0)} \aint_{Q_c(t_0,x_0)} | \rG_r \left[ g_1 +g_2 \right](t,x) - \rG_r \left[g_1+g_2 \right](s,y) | dtdxdsdy \\
&\qquad \leq 2 \cM ( \rG_r [g_1] ) (t_1,x_1)  + \aint_{Q_c(t_0,x_0)} \aint_{Q_c(t_0,x_0)} \bG_r g_2(t,x,s,y) dtdxdsdy 
\end{align*}
\end{lemma}
\begin{proof}
The proof of this lemma is almost identical to that of Lemma \ref{lem20200108}.
We only mention that \eqref{20200108} is used in place of \eqref{2019092801}.

\end{proof}

\begin{lemma}
						\label{lem2020010802}
Let $r \in (1,\infty)$, $c \in (0,\infty)$, $(t_0,x_0) \in \fR^{d+1}$, and $f \in   \bL_r(\infty) \cap \bL_{\infty,r}(\infty)$. 
Assume that $f$ vanishes outside of $Q_{8c}(t_0,x_0)$ $(a.s.)$.
Then there exists a constant $N=N(d,r,\kappa,K)$ such that
\begin{align}
							\label{2019111901}
 \aint_{Q_c(t_0,x_0)}| \cG_r[f](t,x)| dtdx 
 \leq N\| f \|_{\bL_{\infty,r}(\infty)}.
\end{align}
\end{lemma}

\begin{proof}
By Theorem \ref{thm 5-2}, we have
$$
\int_0^\infty \int_{\fR^d}  |\cG_r[f] (t,x)|^r dtdx \leq N  \int_0^\infty \int_{\fR^d} \bE |f(t,x)|^r dtdx.
$$
Therefore recalling Jensen's inequality and the vanishing assumption on $f$, we get
\begin{align*}
 \aint_{Q_c(t_0,x_0)}| \cG_r \left[ f \right](t,x)| dtdx 
&\leq  \left(\aint_{Q_c(t_0,x_0)} | \cG_r \left[ f \right](t,x)|^r dtdx \right)^{1/r} \\
& \leq N \left( \frac{1}{|Q_c(t_0,x_0) | } \int_0^\infty \int_{\fR^d} \bE| f(t,x)|^r dtdx \right)^{1/r} \\
& \leq N \left( \frac{1}{|Q_c(t_0,x_0) | } \int_{Q_{8c}(t_0,x_0)} \bE| f(t,x)|^r dtdx \right)^{1/r} \\
&\leq N  \sup_{(t,x) \in (0,\infty) \times \fR^d} \left(\bE| f(t,x)|^r\right)^{1/r}.
\end{align*}
The lemma is proved. 
\end{proof}

\begin{lemma}
							\label{lem2020010805}
Let $r \in [2,\infty)$, $c \in (0,\infty)$, $(t_0,x_0) \in \fR^{d+1}$, and $g \in   \tilde \bL_r(\infty,l_2) \cap \tilde \bL_{\infty,r}(\infty,l_2)$. 
Assume that $g$ vanishes outside of $Q_{8c}(t_0,x_0)$ $(a.s.)$.
Then there exists a constant $N=N(d,r,\kappa,K)$ such that
\begin{align*}
 \aint_{Q_c(t_0,x_0)}| \rG_r[g](t,x)| dtdx 
 \leq \| g \|_{\bL_{\infty,r}(\infty,l_2)}.
\end{align*}
\end{lemma}
\begin{proof}
The proof of this lemma is almost identical to that of Lemma \ref{lem2020010802}.
We only mention that Theorem \ref{thm 5-4} is used in place of Theorem \ref{thm 5-2}.
\end{proof}

\begin{lemma}
						\label{lem2020010803}
Let $r \in (1,\infty)$, $c \in (0,\infty)$, $(t_0,x_0) \in \fR^{d+1}$, and $f \in   \bL_r(\infty) \cap \bL_{\infty,r}(\infty)$. 
Assume that $f$ vanishes on  $Q_{8c}(t_0,x_0)$ $(a.s.)$.
Then there exists a constant $N=N(d,r,\kappa,K)$ 
\begin{align}
							\label{2019110501}
 \aint_{Q_c(t_0,x_0)}   \aint_{Q_c(t_0,x_0)}| G_r[f](t,x,s,y) | dt  dx ds dy 
  \leq N\| f \|_{\bL_{\infty,r}(\infty)}.
\end{align}
\end{lemma}

\begin{proof}
By using the Sobolev mollifiers used in the proof of Lemma \ref{lem20200108} with the observation that
$\| f^\varepsilon \|_{\bL_{\infty,r}(\infty)} \leq \| f \|_{\bL_{\infty,r}(\infty)}$ for all $\varepsilon>0$,
we may assume that $f \in   \bH^1_r(\infty) \cap \bL_{\infty,r}(\infty)$.
Let $(t,x), (s,y) \in Q_c(t_0,x_0)$. Then by Minkowski's inequality and the vanishing assumption on $f$,
\begin{align*}
&G_r[f](t,x,s,y) \\
&\leq  \left[\int_{ \fR^{d+1} \setminus Q_{8c}(t_0,x_0)}  \sup_{\omega \in \Omega} \left| \bar p_{xx}(\omega,t,\rho , x-z)  - \bar p_{xx}(\omega,s,\rho , y-z)  \right|   dz d\rho \right]\\
&\qquad \times  \sup_{ (\rho,z) \in \fR^{d+1} }  \left(\bE \left[  \left|f(\rho,z)  \right|^r \right] \right)^{1/r} .
\end{align*}
Therefore by \eqref{2020012602} and taking the mean average, we obtain \eqref{2019110501}.
\end{proof}

\begin{corollary}
						\label{2020010610}
Let $r \in (1,\infty)$, $(t_1,x_1) \in \fR^{d+1}$,  and $f_1,f_2 \in   \bL_r(\infty) \cap \bL_{\infty,r}(\infty)$. 
Then there exists a positive constant $N=N(d,\kappa,K,r)$ such that 
\begin{align}
 (\cG_r [f_1+f_2])^\sharp(t_1,x_1) 
						\label{2019111910}	
\leq  2 \cM ( \cG_r [f_1] ) (t_1,x_1) + N\| f_2 \|_{\bL_{\infty,r}(\infty)}
\end{align}
\end{corollary}
\begin{proof}
We may assume that $f_1,f_2 \in   \bH^1_r(\infty) \cap \bL_{\infty,r}(\infty)$.
Let $Q_c(t_0,x_0)$ be a cylinder containing $(t_1,x_1)$. 
Due to \eqref{2019092802}, it is sufficient to show that 
$$
 \aint_{Q_c(t_0,x_0)} \aint_{Q_c(t_0,x_0)} G_r [f_2](t,x,s,y) dtdxdsdy 
 \leq N\| f_2 \|_{\bL_{\infty,r}(\infty)}.
$$
Put 
$$
f_{2,1} :=  1_{ Q_{8c}(t_0,x_0) } \cdot f_2  \quad \text{and} \quad f_{2,2} :=  (1-1_{ Q_{8c}(t_0,x_0) } ) \cdot f_2 .
$$
Then by \eqref{2019092801}, \eqref{2019111901}, and \eqref{2019110501},
\begin{align*}
& \aint_{Q_c(t_0,x_0)} \aint_{Q_c(t_0,x_0)} G_r[f_2](t,x,s,y) dtdxdsdy   \\
&\leq \aint_{Q_c(t_0,x_0)} \cG_r [f_{2,1}](t,x) dt dx + \aint_{Q_c(t_0,x_0)} \cG_r [f_{2,1}](s,y)  dsdy \\
&\qquad \qquad \qquad \qquad + \aint_{Q_c(t_0,x_0)} \aint_{Q_c(t_0,x_0)} G_r[f_{2,2}](t,x,s,y) dt dx ds dy \\
& \leq N \left( \| f_{2,1} \|_{\bL_{\infty,r}(\infty)} +  \| f_{2,2} \|_{\bL_{\infty,r}(\infty)} \right) \\
& \leq N \| f_{2} \|_{\bL_{\infty,r}(\infty)}.
\end{align*}
The corollary is proved. 
\end{proof}

\begin{lemma}
							\label{lem2020010804}
Let $r \in [2,\infty)$, $c \in (0,\infty)$, $(t_0,x_0) \in \fR^{d+1}$, and $g \in   \tilde \bL_r(\infty,l_2) \cap \tilde \bL_{\infty,r}(\infty,l_2)$. 
Assume that $g$ vanishes on  $Q_{8c}(t_0,x_0)$ $(a.s.)$.
Then there exists a constant $N=N(d,r,\kappa,K)$ 
\begin{align*}
 \aint_{Q_c(t_0,x_0)}   \aint_{Q_c(t_0,x_0)}| \bG_r[g](t,x,s,y) | dt  dx ds dy 
  \leq N\| g \|_{\bL_{\infty,r}(\infty,l_2)}.
\end{align*}
\end{lemma}
\begin{proof}
The proof of this lemma is almost identical to that of Lemma \ref{lem2020010803}.
We only mention that Corollary \ref{cor hor 2} is used in place of \eqref{2020012602}.
\end{proof}

\begin{corollary}
						\label{2020010621}
Let $r \in [2,\infty)$, $(t_1,x_1) \in \fR^{d+1}$, and $g_1, g_2 \in  \tilde  \bL_r(\infty) \cap \tilde \bL_{\infty,r}(\infty, l_2)$.
Then there exists a positive constant $N=N(d,\kappa,K,r)$ such that 
\begin{align*}
 (\rG_r [g_1+g_2])^\sharp(t_1,x_1) 
\leq  2 \cM ( \rG_r [g_1] ) (t_1,x_1) + N\| g_2 \|_{\bL_{\infty,r}(\infty,l_2)}
\end{align*}
\end{corollary}
\begin{proof}
Follow the proof of Corollary \ref{2020010610} with Lemmas \ref{lem2020010805} and \ref{lem2020010804}.
\end{proof}
To handle the boundedness of $\cG_r$ and $\rG_r$ simultaneously, we  introduce a Banach space valued operator. 
Let $\rB$ be a Banach space and $\bL$ be a subspace of $L_r(\fR^{d+1} ; \rB) $. Recall that any operator $\cL$ from $\bL$ to $L_r( \fR^{d+1})$ is called quasilinear operator if there exists a positive constant $N_0$ such that for all $c \in \fR$ and $f_1,f_2 \in \bL$, we have
$$
\cL (cf_1) = c \cL f_1,
$$
$$
\cL ( f_1 + f_2) \leq N_0 \left(\cL f_1 + \cL  f_2 \right).
$$
Note that both operators $f \mapsto (\cG_r[f])^\sharp$ and $g \mapsto (\rG_r[g])^\sharp$ are not quasilinear even though
$f \mapsto \cG_r[f] $ and $g \mapsto \rG_r[g]$ are.  
Thus the classical Marcinkiewicz interpolation theorem cannot be applied directly and we need the following variant of the interpolation theorem.
We emphasize that the operator $\cL$ appearing in the following lemma does not have to be quasilinear. 
\begin{lemma}
						\label{ext lem}
Let $1 < r \leq p <\infty$, $\rB$ be a Banach space, $\bL$ be a subspace of $L_r(\fR^{d+1} ; \rB) $,  and 
$\cL$ be a bounded operator from $\bL$ to $L_r( \fR^{d+1})$, i.e. 
 there exists a positive constant $N$ such that
\begin{align}
							\label{2020010602}
\| \cL f \|_{L_r( ( \fR^{d+1})} \leq N_1 \| f \|_{ L_r(  \fR^{d+1} ;\rB) } 
\end{align}
for all $ f \in  \bL$.
Assume that there exists a positive constant $N_1$ such that for all $(t,x) \in \fR^{d+1}$, $\lambda>0$,  and $f \in \bL$,
\begin{align}
							\label{20200106}
 (\cL [f])^\sharp(t,x) 
\leq   N_2 \left(\cM ( \cL [f_{1,\lambda}] ) (t,x) + \|f_{2,\lambda} \|_{L_\infty((0,\infty) \times \fR^d; \rB)} \right),
\end{align}
where
$$
f_{1,\lambda}(t,x) =  f(t,x) 1_{ \{ \| f\|_{\rB}  > \lambda \} } (t,x)
$$
and
$$
f_{2,\lambda}(t,x) =  f(t,x) 1_{ \{ \| f\|_{\rB}   \leq  \lambda \} } (t,x).
$$
Then there exists a constant $N=N(d,p,N_1,N_2)$ such that 
\begin{align}
							\label{2020010630}
\| \cL f \|_{L_p( ( \fR^{d+1})} \leq N \| f \|_{ L_p( (0,\infty) \times \fR^d ;\rB) }
\end{align}
for all $f \in   \bL$.
\end{lemma}
\begin{proof}
If $p=r$, then the result is trivial. So we assume $r < p$.
Let $f \in \bL$.
For each $\lambda>0$ and $\delta>0$, we decompose $f$ into
$$
f_{1, \delta \lambda}(t,x) =  f(t,x) 1_{ \{ \| f\|_{\rB}  > \delta \lambda \} } (t,x)
$$
and
$$
f_{2, \delta \lambda}(t,x) =  f(t,x) 1_{ \{ \| f\|_{\rB}   \leq  \delta \lambda \} } (t,x),
$$
where $\delta>0$ will be fixed later. 
Let $\lambda \leq  (\cL [f] )^\sharp(t,x)$.
Then by \eqref{20200106},
\begin{align*}
\lambda \leq (\cL [f] )^\sharp(t,x)
&\leq N_2 \cM ( \cL [f_{1,\delta \lambda}] ) (t,x) + N_1\| f_{2, \delta \lambda} \|_{L_\infty((0,\infty) \times \fR^d; \rB)} \\
&\leq  N_2 \cM ( \cL  [f_{1, \delta \lambda}] ) (t,x) + N_2\delta \lambda.
\end{align*}
Fixing $\delta>0$ so that $N_2 \delta < \frac{1}{2}$, we have
$$
\lambda \leq 2N_2  \cM ( \cL  [f_{1, \delta \lambda}] ) (t,x).
$$
Thus for each $\lambda >0$,
\begin{align*}
|\{(t,x) \in \fR^{d+1} : \lambda \leq (\cL  [f])^\sharp (t,x) \} | 
\leq |\{(t,x) \in \fR^{d+1} : \lambda \leq  2N_2  \cM ( \cL [f_{1, \delta \lambda}] ) (t,x) \} | .
\end{align*}
Moreover by \eqref{fs thm} and \eqref{2020010602},
$$
\|   \cM (\cL  [f_{1, \delta \lambda}] ) \|_{L_p(\fR^{d+1})}
\leq N\| \cL  [f_{1, \delta \lambda}]  \|_{L_p(\fR^{d+1})}
\leq N\| f_{1, \delta \lambda}  \|_{L_p(\fR^{d+1})}
$$
for all $\lambda > 0$.
Therefore by Fubini's Theorem and Chebyshev's inequality,
\begin{align*}
\| (\cL [f] )^\sharp \|^p_{L_p(\fR^{d+1})}
&=  p \int_0^\infty \lambda^{p-1} |\{(t,x) \in \fR^{d+1} : \lambda \leq (\cL [f])^\sharp (t,x) \} |  d\lambda  \\
&\leq  p \int_0^\infty \lambda^{p-1} |\{(t,x) \in \fR^{d+1} : \lambda \leq 2N_2  \cM ( \cL [f_{1,\delta \lambda}] ) (t,x) \} | d\lambda  \\
&\leq  N  \int_0^\infty \lambda^{p-r-1}  \int_{\fR^{d+1}}  |\cM ( \cL [f_{1, \delta \lambda}] ) (t,x)|^r  dt dx d\lambda  \\
&\leq  N  \int_0^\infty \lambda^{p-r-1}  \int_{ (0,\infty) \times \fR^{d}}  \| f_{1, \delta \lambda} (t,x) \|_{\rB}^r  dt dx d\lambda  \\
&=  N  \int_{ (0,\infty) \times \fR^d}  \int_0^{  \|f\|_{\rB}  / \delta} \lambda^{p-r-1}   d \lambda  \| f (t,x) \|_{\rB}^r dt dx  \\
&\leq   N  \int_{ (0,\infty) \times \fR^{d}}   \| f(t,x)\|^p_{\rB}  dt dx.
\end{align*}
By \eqref{fs thm} again,
$$
\|  \cL [f] \|_{L_p(\fR^{d+1})}\leq N\| (\cL [f] )^\sharp \|_{L_p(\fR^{d+1})},
$$
which completes the proof  of the lemma. 
\end{proof}

\vspace{3mm}
{\bf Proof of Theorem \ref{thm 5}}

\vspace{3mm}

Let $r \in (1,\infty)$ and  $f \in \bL_{r}(\infty) \cap \bL_{\infty}(\infty)$. We apply Lemma \ref{ext lem}.
Recall
$$
\bL_{r}(\infty)  = L_r( (0,\infty) \times \fR^d ; L_r(\Omega))
$$
and
$$
\bL_{\infty}(\infty)  =  L_\infty( (0,\infty) \times \fR^d ; L_r(\Omega))
$$
By properties of $L_p$-space, $L_r (\Omega)$ is a Banach space.
Obviously, $\bL_{r}(\infty) \cap \bL_{\infty}(\infty) $ is a subspace of $  L_r( \fR^{d+1} ; L_r(\Omega))$
by considering the trivial extensions with the indicator $1_{(0,\infty)}$.
Thus  by setting $\bL:= \bL_{r}(\infty) \cap \bL_\infty ( \infty)$,
the operator $\cG_r$ is an operator from  $\bL$ to $L_r(\fR^{d+1})$. 
 Moreover,  by Theorem \ref{thm 5-2} and  Corollary \ref{2020010610}, the main assumptions \eqref{2020010602} and \eqref{20200106} hold
  with $\cL = \cG_r$ and $\rB = L_r(\Omega)$.
Finally the theorem is proved due to \eqref{2020010630}.

\vspace{3mm}
{\bf Proof of Theorem \ref{thm 5-3}}

\vspace{3mm}

The proof of this theorem is also an easy application of Lemma \ref{ext lem}.
Noe that $L_r (\Omega;l_2)$ is a Banach space and the operator $\rG_r$ is an operator from 
$\bL:= \tilde \bL_{r}(\infty,l_2) \cap \tilde \bL_\infty ( \infty,l_2)$ to $L_r(\fR^{d+1})$. 
Moreover, Theorem \ref{thm 5-4} and Corollary \ref{2020010621} imply that \eqref{2020010602} and \eqref{20200106} hold 
 with $\cL = \rG_r$ and $\rB = L_r(\Omega;l_2)$. 
 Finally the theorem is obtained from \eqref{2020010630}.

\mysection{the proofs of main theorems}
									\label{pf thm 3-2}

\vspace{3mm}
{\bf Proof of Theorem \ref{thm 3-2}}

\vspace{3mm}
The uniqueness comes from the classical result (see \cite[Theorem 4.2]{Krylov1999}).
Thus we focus on proving the existence of a solution $u$ satisfying \eqref{2020011001} and \eqref{2020011001-2}.
We use the solution representation and boundedness of the operators $\cT$ and $\cG_r$.
As mentioned in the proof of Theorem \ref{thm 3-1}, for each $\omega \in \Omega$, there exists a unique solution $u(t,x)$ to \eqref{det ran eqn} such that
\begin{align}
								\label{20200124}
u(t,x) = \int_0^t \int_0^t \bar p(\omega,t,s,x-y)f(s,y) dyds= \cT f(t,x),
\end{align}
where $\bar p(\omega, t,s,x-y)$ is defined in \eqref{ran ker def}. 
Thus it is sufficient to show \eqref{2020011001} and \eqref{2020011001-2} for $u$ defined in \eqref{20200124}. 
First we consider very nice stochastic process $f$, i.e. assume that $f \in  \bH_c^\infty(T)$.
Then  
\begin{align*}
\bE\left[|u_{xx}(t,x)|^r \right]^{1/r} 
&= \left( \bE\left[ \left|\int_0^t \left[ \int_{\fR^d} \bar p_{xx}(\omega,t,s,x-y)f(s,y) dy \right] ds \right|^r \right]  \right)^{1/r} \\
&= \cG_r[f] (t,x)
\end{align*}
and
$$
f \in \bL_{\infty,r}(\infty) \cap \bL_r(\infty).
$$
Therefore, by Lemma \ref{lem20200110} and Theorem \ref{thm 5}, we have \eqref{2020011001} and \eqref{2020011001-2} for all
$f \in  \bH_c^\infty(T)$.
Next we consider the general case $f \in \bL_{p,r}(T)$. We use the standard approximation argument based on the linear property of the equation. 
By Lemma \ref{lem dense}, $ \bH_c^\infty(T)$ is dense in $\bL_{p,r}(T)$. Thus there exists a sequence $f_n \in \bH_c^\infty(T)$ such that
$$
\|f_n - f\|_{\bL_{p,r}(T)} \to 0 \quad \text{as} \quad n \to \infty. 
$$
For each $f_n$, there exists a unique solution $u_n \in \bH_{p,r}^2(T)$ to \eqref{det ran eqn} and thus
$u_n -u_m$ becomes a solution to 
\begin{align*}
				 d(u_n-u_m)= \left( a^{ij}(\omega, t)(u_n -u_m)_{x^ix^j}+ f_n -f_m  \right)dt , \quad t \in (0,T); \quad u(0,\cdot)=0. 
\end{align*}
Since we have already obtained \eqref{2020011001} and \eqref{2020011001-2} for all $f_n -f_m \in \bH_c^\infty(T)$, we have
$$
\|u_n -u_m\|_{\bH_{p,r}^2(T)} \leq N_1 \|f_n -f_m\|_{\bL_{p,r}(T)} \qquad 
$$
and
$$
\|(u_n)_{xx} -(u_m)_{xx}\|_{\bL_{p,r}(T)} \leq N \|f_n -f_m\|_{\bL_{p,r}(T)} \qquad 
$$
for all $n, m$. 
Finally the uniqueness of a solution and the completeness of $L_p$-spaces show that \eqref{2020011001} and \eqref{2020011001-2} hold for all $f \in \bL_{p,r}(T)$ and the corresponding solution  $u$. The theorem is proved. \qed

\vspace{3mm}
{\bf Proof of Theorem \ref{main thm}}

\vspace{3mm}

Since the uniqueness of a solution $u$ easily comes from  Theorem \ref{thm 3-2}, we only focus on showing the existence of a solution and estimates \eqref{main est} and \eqref{main est-2}.
The proof of the theorem is similar to that of Theorem \ref{thm 3-2}. However, it should be noticed that the solution representation is impossible if the coefficients $a^{ij}(\omega,t)$ are random
since for each $s<t$, $p(\omega,t,s,x-y) g(s,y)$ is not  $\rF_s$-adapted (cf. Remark \ref{rmk 20200123} and \cite[Remark 4.2]{kim2019sharp}). Thus we need to consider the simplest non-random coefficients $a^{ij}(\omega,t)= \delta^{ij}$ first,  where $\delta^{ij}$ denotes the Kronecker delta.

\vspace{2mm}
{\bf Step 1}. We assume that $f=0$, $g \in \tilde \bH_c^\infty(T,l_2)$, and $a^{ij}(\omega,t) =\delta^{ij}$. Then it is well-known (cf. the proof of Theorem 4.2 of \cite{Krylov1999}) that the solution $u$ to \eqref{main eqn} is given by
$$
u(t,x) = \int_0^t \int_{\fR^d} p(t,s,x-y)g^k(s,y)dydw^k_s= \bT g (t,x).
$$
Moreover,  by the Burkholder-Davis-Gundy inequality,
\begin{align*}
\bE\left[|u_{xx}(t,x)|^r\right] 
=  \bE\left[\left|\int_0^t \left[ \int_{\fR^d} p_x(t,s,x-y)g_x^k(s,y)dy \right]dw^k_s \right|^r\right]
\approx |\rG_r[g_x](t,x)|^r,
\end{align*}
that is, there exists a positive constant $N$ such that
\begin{align*}
N^{-1} |\rG_r[g_x] (t,x)|^r \leq \bE\left[\left|\int_0^t \left[ \int_{\fR^d} p_x(t,s,x-y)g_x^k(s,y)dy \right]dw^k_s \right|^r\right]
\leq N |\rG_r[g_x](t,x)|^r,
\end{align*}
where $N=N(d,r)$. 
Thus  by Lemma \ref{lem2020011010} and Theorem \ref{thm 5-3}, we have
\begin{align*}
\| u \|_{\bL_{p,r}(T,l_2)} \leq N_1(d,p,T) \left(  \|g\|_{\bL_{p,r}(T,l_2)}  \right)
\end{align*}
and
\begin{align*}
\| u_{xx} \|_{\bL_{p,r}(T,l_2)} \leq N_2(d,p,r) \left(    \|g_x\|_{\bL_{p,r}(T,l_2)}  \right).
\end{align*}

\vspace{2mm}
{\bf Step 2}. We assume that $f=0$, $g \in \tilde \bH_{p,r}^1(T,l_2)$, and $a^{ij}(\omega, t) =\delta^{ij}$. 
Since $\tilde \bH_c^\infty(T,l_2)$ is dense in $\tilde \bH_{p,r}^1(T,l_2)$, the approximation argument in the proof of Theorem \ref{thm 3-2} gives the result.

\vspace{2mm}
{\bf Step 3} (General case). By Step 2, there exists a unique solution $u^1 \in \bH_{p,r}^2(T)$ to 
\begin{align*}
				 du^1= \left( \Delta u^1  \right)dt + g^k dw^k_t, \quad t \in (0,T); \quad u^1(0,\cdot)=0. 
\end{align*}
such that  
\begin{align}
						\label{2020011030}
\| u^1 \|_{\bL_{p,r}(T)} \leq N_1(d,p,T) \left(  \|g\|_{\bL_{p,r}(T,l_2)}  \right)
\end{align}
and
\begin{align}
						\label{2020011031}
\| u^1_{xx} \|_{\bL_{p,r}(T)} \leq N_2(d,p,r) \left(    \|g_x\|_{\bL_{p,r}(T,l_2)}  \right).
\end{align}
Moreover by Theorem \ref{thm 3-2}, there exists a unique solution $u^2 \in \bH_{p,r}^2(T)$ to 
\begin{align*}
				 du^2= \left(  a^{ij}(\omega,t) u^2_{x^ix^j}  + a^{ij}(\omega,t) u^1_{x^ix^j}- \Delta u^1 + f    \right)dt, \quad t \in (0,T); \quad u^2(0,\cdot)=0
\end{align*}
such that
\begin{align}
							\notag
\| u^2 \|_{\bH^2_{p,r}(T)} 
&\leq N_1(d,p,r,\kappa,K,T) \left(  \| a^{ij}(\omega,t) u^1_{x^ix^j}- \Delta u^1 + f \|_{\bL_{p,r}(T)}  \right) \\
							\label{2020011032}
&\leq N_1 \left(  \|u^1_{xx}\|_{\bL_{p,r}(T)} +\| f \|_{\bL_{p,r}(T)}  \right)
\end{align}
and
\begin{align}
							\notag
\| u^2_{xx} \|_{\bL_{p,r}(T)} 
&\leq N_2(d,p,r,\kappa,K) \left(    \| a^{ij}(\omega,t) u^1_{x^ix^j}- \Delta u^1 + f \|_{\bL_{p,r}(T)}  \right) \\
						\label{2020011033}
&\leq N_2 \left(  \|u^1_{xx}\|_{\bL_{p,r}(T)} +\| f \|_{\bL_{p,r}(T)}  \right).
\end{align}
Thus finally setting $u:=u^1+u^2$ and combining \eqref{2020011030} - \eqref{2020011033}, we have a solution $u \in \bH_{p,r}^2(T)$ to \eqref{main eqn} satisfying
\eqref{main est} and \eqref{main est-2}.  
The theorem is proved. \qed

\vspace{3mm}
{\bf Proof of Theorem \ref{main thm 2}}
\vspace{3mm}

Note that if $p \geq 2$, then the function $|x|^p$ is twice continuous differentiable. 
Thus by Theorem \ref{main thm} and H\"older's inequality, it is obvious that 
\begin{align*}
& \int_0^T \int_{\fR^d} |m_r(t,x)|^{p/r} dt  + \int_0^T \int_{\fR^d} |\partial_{xx} m_r(t,x)|^{p/r} dt \\
&\leq N_1\left( \|f\|^p_{\bL_{p,r}(T)} + \|g\|^p_{\bH^1_{p,r}(T)}\right).
\end{align*}
Therefore it is sufficient to control
\begin{align*}
\int_0^T \int_{\fR^d} |\partial_t m_r(t,x)|^{p/r} dt.
\end{align*}
Let $x \in \fR^d$. 
Recall that $\phi \in C_c^\infty(\fR^d)$ is nonnegative function such that
$\int_{\fR^d}\phi(y)dy = 1$ and $\phi^\varepsilon (y) = \frac{1}{\varepsilon^d} \phi( y/\varepsilon)$.
Putting $\phi^\varepsilon(x- \cdot)$ in \eqref{def sol-2}, we have
\begin{align}
							\notag
u^\varepsilon(t,x)
&=   (u_0 ,\phi) + \int^t_0  \left[  a^{ij}(t) u^\varepsilon_{x^ix^j}(s,x) + (f^\varepsilon(s,x) \right]ds \\
							\label{2020111510}
& +\sum_k \int^t_0  g^\varepsilon(s,\cdot) dw^k_t \qquad \forall t \in (0,T)~ (a.s.),
\end{align}
where $u^\varepsilon(t,x)= \left(u(t,\cdot),\phi^\varepsilon(x- \cdot) \right)$.
Then by It\^o's formula, 
\begin{align*}
&|u^\varepsilon(t,x)|^r  \\
&= \int_0^t r|u^\varepsilon(s,x)|^{r-2}u^\varepsilon(s,x) \left(\bar a^{ij}(s) u^\varepsilon_{x^ix^j}(s,x) + f^\varepsilon(s,x) \right)ds \\
&\quad + \int_0^t r|u^\varepsilon(s,x)|^{r-2}u^\varepsilon(s,x) \left( g^{\varepsilon, k}(s,x) \right)dw^k_t\\
&\quad+ \int_0^t \frac{1}{2} r(r-1)|u^\varepsilon(s,x)|^{r-2}|g^\varepsilon|_{l_2}^2(s,x) ds
\end{align*}
for all $t \in (0,T)$ almost surely.
Moreover, due to \eqref{2020111510} and H\"older's inequality,
\begin{align*}
&\int_0^t |u^\varepsilon(s,x)|^{2(r-1)} |g^{\varepsilon }(s,x)|_{l_2}^2 ds \\
&\leq N\sup_{s \leq t} | u^\varepsilon(s,x) |^{2(r-2)} \int_0^t \int_{U(x)}|g(s,y)|_{l_2}^2dy ds \\
&\leq N(\omega)\int_0^t \int_{U(x)}|g(s,y)|_{l_2}^2dy ds \\
\end{align*}
for all $t \in (0,T)$ almost surely, where $U(x)$ is a bounded subset of $\fR^d$ depending on $x$. 
Since
\begin{align*}
\bE\left[ \left|\int_0^t \int_{U(x)}|g(s,y)|_{l_2}^2dy ds\right|^{r/2} \right]
\leq N \|g\|^{r/2}_{\bL_{p,r(T,l_2)}} <\infty,
\end{align*}
we obtain
\begin{align*}
\bE\left[\int_0^t r|u^\varepsilon(s,x)|^{r-2}u^\varepsilon(s,x) \left( g^{\varepsilon, k}(s,x) \right)dw^k_t \right]=0.
\end{align*}
Thus we have
\begin{align*}
&m^\varepsilon_r(t,x):=\bE|u^\varepsilon(t,x)|^r  \\
&= \int_0^t r\bE \left[|u^\varepsilon(s,x)|^{r-2}u^\varepsilon(s,x) \left(\bar a^{ij}(s) u^\varepsilon_{x^ix^j}(s,x) + f^\varepsilon(s,x) \right) \right]ds \\
&\quad+ \int_0^t \frac{1}{2} r(r-1)\bE \left[|u^\varepsilon(s,x)|^{r-2}|g^\varepsilon|_{l_2}^2(s,x)  \right] ds.
\end{align*}
Applying the chain rule,  H\"older inequality, and Young inequality with any $\delta>0$, we have
\begin{align*}
& \partial_t m^\varepsilon_r(t,x)\\
&=  r\bE \left[|u^\varepsilon(t,x)|^{r-2}u^\varepsilon(t,x) \left(\bar a^{ij}(t) u^\varepsilon_{x^ix^j}(t,x) + f^\varepsilon(t,x) \right) \right] \\
&\quad+    \frac{1}{2} r(r-1) \bE \left[|u^\varepsilon(t,x)|^{p-2}|g^\varepsilon|_{l_2}^2(t,x)  \right] \\
& \leq \delta \bE\left[ \left(u^\varepsilon(t,x)\right)^r \right]  \\
&\quad + N\left( \bE\left[ \left(u_{xx}^\varepsilon(t,x)\right)^r \right]  
+\bE\left[ \left(f^\varepsilon(t,x)\right)^r \right]
+\bE\left[ \left(g^\varepsilon(t,x)\right)^r \right]\right).
\end{align*} 
Taking $\delta>0$ small enough we have
\begin{align*}
\int_0^T \int_{\fR^d} |\partial_t m_r(t,x)|^{p/r}  dt \leq N_1\left( \|f\|^p_{\bL_{p,r}(T)} + \|g\|^p_{\bH^1_{p,r}(T)}\right).
\end{align*}
Finally taking $\varepsilon \to 0$ and applying Fatou's lemma, we obtain \eqref{main est 2}.

\mysection{Acknowledgement}

The author is very grateful to prof. Kyeong-Hun Kim for helpful discussions and suggesting Theorem \ref{main thm 2} as an application.

\end{document}